\documentclass[12pt]{amsart}
\usepackage{amsfonts}
\usepackage{amssymb}
\usepackage[utf8x]{inputenc}
\usepackage{xspace}
\usepackage[inline,shortlabels]{enumitem}
\usepackage[a4paper,margin=3cm]{geometry}
\usepackage{mathtools}
\usepackage{xfrac}
\usepackage{breqn}
\usepackage{svn}

\SVN $Revision: 74 $

\usepackage[colorlinks,anchorcolor=black,citecolor=black,linkcolor=black]{hyperref}

\usepackage{booktabs}
\usepackage[bb=boondox]{mathalfa}

\newtheorem{theorem}{Theorem}[section]
\newtheorem{corollary}[theorem]{Corollary}

\newtheorem{lemma}[theorem]{Lemma}
\newtheorem{proposition}[theorem]{Proposition}

\theoremstyle{definition}
\newtheorem{example}[theorem]{Example}
\newtheorem{examples}[theorem]{Examples}
\newtheorem{definition}[theorem]{Definition}
\newtheorem{remark}[theorem]{Remark}
\newtheorem{remarks}[theorem]{Remarks}

\def\Cc{\hbox{\sf C\kern -.47em {\raise .48ex \hbox{$\scriptscriptstyle |$}}
		\kern-.5em {\raise .48ex \hbox{$\scriptscriptstyle |$}} }}

\newcommand{\mc}[1]{\mathcal {#1}}

\newcommand{\mcb}[1]{\widehat{\mathcal {#1}}}
\newcommand{\Sn}[1]{\mc{#1}_n}
\newcommand{\Snb}[1]{\mcb{#1}_n}

\newcommand{\PB}{\Snb{M}} 
\renewcommand{\PB}{\Sn{B}}

\newcommand{\DD}{\Sn{D}} 
\newcommand{\PBDD}{\Sn{S}} 

\newcommand{\UT}{\Sn{R}} 
\newcommand{\LT}{\Sn{L}} 

\newcommand{\PBUT}{\Snb{R}} 
\newcommand{\PBLT}{\Snb{L}} 

\newcommand{\CC}{\mathbb{C}} 

\newcommand{\EE}{\mathbb{E}}
\newcommand{\LL}{\mathbb{L}}
\newcommand{\MM}{\mathbb{M}}

\DeclareMathOperator{\tr}{\mathrm{tr}}
\DeclareMathOperator{\trace}{\mathrm{Tr}}

\newcommand{\Cnn}{\CC^{n \times n}}
\renewcommand{\Cnn}{\Sn{M}}

\numberwithin{equation}{section}

\newcommand{\addconv}{\boxplus}
\newcommand{\multconv}{\boxtimes}
\newcommand{\act}[1]{{#1}^{\,\addconv}}
\newcommand{\mct}[1]{{#1}^{\,\multconv}}
\newcommand{\biact}[1]{{#1}^{\,\addconv \, \addconv}}

\newcommand{\ltfrac}[2]{\mbox{\large$\frac{#1}{#2}$}}

\usepackage{color}

\mathchardef\mhyphen="2D
\begin{document}

	\title[]{On matrices in finite free position}
	
	\author{Octavio Arizmendi}
	\email[O.~Arizmendi]{octavius@cimat.mx}
	\author{Franz Lehner}
	\email[F.~Lehner]{lehner@math.tugraz.at}
	\thanks{\mbox{\includegraphics[height=2em]{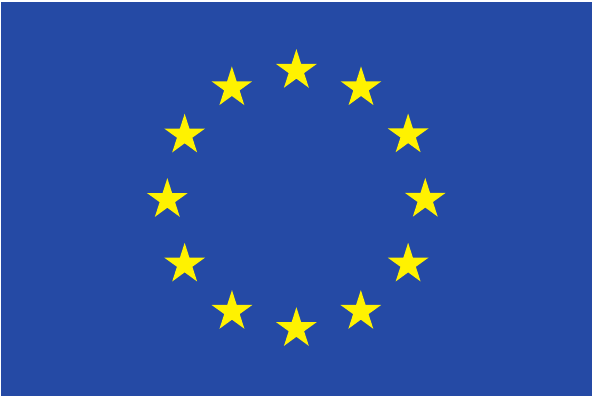}} \parbox[b][1.5em]{0.25\textwidth}{Co-funded
			by\\ the European Union}
		\parbox[b][1.5em]{0.65\textwidth}{F.L.~was partly supported by the H2020-MSCA-RISE project 734922 - CONNECT}
	}
	
	\author{Amnon Rosenmann}
	\email[A.~Rosenmann]{rosenmann@math.tugraz.at}
	\thanks{A.R.~was supported by the Austrian Science Fund (FWF) project p29355}
	
	\date{
		Rev.\SVNRevision, 
		\today}
	
	\begin{abstract}
		We study pairs $(A,B)$ of square matrices that are in additive (resp.\
		multiplicative) finite free position, that is, the characteristic polynomial $\chi_{A+B}(x)$ (resp.\ $\chi_{AB}(x)$) equals the additive finite free convolution $\chi_{A}(x) \addconv \chi_{B}(x)$ (resp. the multiplicative finite free convolution $\chi_{A}(x) \multconv \chi_{B}(x)$),
		which equals the expected characteristic polynomial $\mathbb{E}_U \, [ \chi_{A+U^* BU}(x) ]$ (resp.\ $\mathbb{E}_U \, [ \chi_{AU^* BU}(x) ]$) over the set of unitary matrices $U$.
		We examine the lattice of (non-irreducible) affine algebraic sets of matrices consisting of finite free complementary pairs with respect to the additive (resp.\ multiplicative) convolution.
		We show that these pairs include the diagonal matrices vs.~the principally balanced matrices, the upper (lower) triangular matrices vs.~the upper (lower) triangular matrices with constant diagonal,
		and the scalar matrices vs.~the set of all square matrices.
	\end{abstract}

	\maketitle
	\tableofcontents{}
	\vspace{-15mm}
	\noindent
	\\
	\section{Introduction}
        Finite free probability was introduced by Marcus, Spielman and Srivastava \cite{MSS22} and Marcus \cite{Marcus21} in connection with the solution of the Kadison--Singer conjecture.
        They introduced additive and multiplicative convolutions of characteristic polynomials of matrices over the real and complex numbers. In fact, these convolutions were already studied in the 1920s by Walsh \cite{Walsh22} and Szeg\H{o} \cite{Szego22}, who proved among other results that under certain conditions they preserve real rootedness.
		In \cite{MSS22} explicit formulas for the convolutions were
                introduced, showing that they coincide with  the expectations of the characteristic polynomials of unitarily invariant random matrices, thus linking them to free probability.
		That is, when $A$ and $B$ are normal matrices then the additive convolution of their characteristic polynomials equals the expectation of the characteristic polynomials of $A+U^*BU$, where the expectation is over randomly distributed unitary matrices $U$.
		A similar result holds for normal matrices with respect to multiplicative convolution.
	
		The analogues of these convolutions in the setting of tropical algebra were studied in \cite{RLP19}.
	
		In \cite{Marcus21} Marcus developed further the findings of \cite{MSS22} to reach a new theory of ``finite free probability'', a finite version of free probability that was introduced by Voiculescu \cite{Voicu91}, with many analogous properties. The approach used by Marcus was mainly analytic, through several transforms that are the analogues of the transforms applied in free probability. An alternative combinatorial approach via set partitions was chosen by Arizmendi et al.\  \cite{AP18, AGP21} who established an analogue of Speicher's combinatorial approach to free probability \cite{Speicher94}.
	
		In the present paper we study pairs of complex matrices $A,B$ which are in finite free position (FFP), that is, such that the characteristic polynomial of the sum satisfies the identity $\chi_{A+B}(x)=\chi_{A} \addconv \chi_{B}$, the additive finite free convolution of $\chi_{A}$ and $\chi_{B}$, which further equals the expected characteristic polynomial $\EE_U \, [ \chi_{A+U^* BU}(x) ]$ over the set of unitary matrices $U$. We consider the analogous question in the multiplicative case.
		
%
		The paper is organized as follows.

        In Section~\ref{sec:addFFP} we recall the definition of additive finite free convolution and FFP.

        In Subsection~\ref{subsec:2x2} we study in detail the simplest case of matrices of dimension $2\times 2$ that are in additive FFP and in Subsection~\ref{subsec:3x3} we treat matrices of dimension $3\times 3$ . 

		In Section~\ref{sec:addlattice} we change the point of view from single matrices to the lattice of (non-irreducible) affine algebraic sets of matrices $\mc{V}$ and $\mc{W}$, which are maximal with respect to the property that each matrix in $\mc{V}$ is in additive FFP with each matrix in $\mc{W}$. We say that such families form a finite free complementary pair. We give an upper bound for the additive finite free rank of each such affine algebraic set.

		In Section~\ref{sec:addcomp} we focus on specific complementary pairs: diagonal and principally balanced matrices (matrices with the property that for each $i$ the values of all principal minors of size $i$ coincide); upper triangular and upper triangular matrices with constant diagonal; scalar matrices and the set of all square matrices.

		In Section~\ref{sec:moments}, we examine moments and cumulants of matrices $A$ and $B$ that are in additive FFP and show that the
		moments of $A+B$ can be expressed in terms of the moments of $A$ and $B$, and compute explicit formulas for the first four moments. We also show that a matrix $A$ is in additive FFP with itself if and only if 
		it is the multiple of a unipotent matrix, i.e., the sum of a nilpotent and a scalar matrix.
	
		In Section~\ref{sec:multconv} we extend our analysis to matrices in the
		multiplicative FFP. In particular, we show that the three complementary pairs of (non-irreducible) affine algebraic sets mentioned above form complementary pairs also in the multiplicative case. In general, however, the pairs of matrices that are in additive FFP and those in multiplicative FFP are not the same.
	
	\section{Matrices in additive finite free position}
	\label{sec:addFFP}
        Additive convolution of polynomials can be  defined as follows \cite{MSS22}.
	\begin{definition}
		Let $p(x) = \sum_{k=0}^{n} a_k x^{n-k}$, $q(x) = \sum_{k=0}^{n} b_k x^{n-k}$ be two polynomials of degree $n$ over $\CC$.
		The \textbf{additive convolution} of $p(x)$ and $q(x)$, denoted $p(x) \addconv q(x)$, is
		\begin{align}
			\label{eq:addconv}
			p(x) \addconv q(x)
			&:= \sum_{k=0}^{n} \left(\sum_{i+j=k} \frac{{n-i \choose j}}{{n \choose j}} a_i b_j \right) x^{n-k} \nonumber\\
			&= \frac{1}{n!} \sum_{k=0}^{n} \frac{1}{(n-k)!} \left(\sum_{i+j=k} (n-i)!(n-j)! a_i b_j \right) x^{n-k} \\
			&= \frac{1}{n!} \sum_{k=0}^{n} p^{(k)}(x) q^{(n-k)}(0)
			= \frac{1}{n!} \sum_{k=0}^{n} p^{(k)}(0) q^{(n-k)}(x), \nonumber
		\end{align}
		where we denote by $p^{(k)}$ the $k$-th derivative of $p$ and by $q^{(n-k)}$ the $(n-k)$-th derivative of $q$.
	\end{definition}
	
	\begin{example}
		\label{ex:unit}
		The neutral element with respect to the operation of additive convolution on polynomials of degree $n$ is $x^n$: when $p(x)$ is of degree $n$ and $q(x)=x^n$ then
		$$
		p(x) \addconv q(x) = \frac{1}{n!} \sum_{k=0}^{n} p^{(k)}(x) q^{(n-k)}(0)
		= \frac{1}{n!} p^{(0)}(x) q^{(n)}(0) = \frac{p(x)n!}{n!}=p(x),
		$$
		and similarly, $x^n\addconv p(x) = p(x)$.
	\end{example}
	
	We denote by $\Cnn = \CC^{n \times n}$ the set of $n \times n$ matrices over $\CC$.
	Given a matrix $A \in \Cnn$, its characteristic polynomial is $\chi_A(x) := \det(xI-A)$.
	
	Following is one of the main results of \cite{MSS22}.
	Recall that a signed permutation matrix is a square matrix with each row and each column having exactly one non-zero entry which is either $1$ or $-1$.
\begin{theorem}\cite{MSS22}
  \label{thm:addorthogonal}
  Let $A,B \in \Cnn$ be normal matrices.
  Then 
  \begin{equation}
    \label{eq:addorthogonal}
    \chi_{A}(x) \addconv \chi_{B}(x)
    = \int_{\mc{U}(n)}     \chi_{A+U^* BU}(x) \,dU
    = \frac{1}{2^nn!}\sum_{P\in \mc{P}^{\pm}(n)} \chi_{A+P^{T} BP}(x),
  \end{equation}
  where the expectation is taken over the set of unitary matrices $\mc{U}(n)$ or the signed permutation matrices $\mc{P}^\pm(n)$.
	\end{theorem}
	In the present paper we study pairs of matrices $A$ and $B$ which satisfy equality \eqref{eq:addorthogonal}  without taking the expected values of conjugations of $B$.
	\begin{definition}
		The matrices $A,B \in \Cnn$ are in \textbf{additive finite free position} (or in additive FFP), if 
		$$
		\chi_{A+B}(x) = \chi_{A}(x) \addconv \chi_{B}(x). 
		$$
		Two families $\mc{A},\mc{B} \subseteq \Cnn$ are in additive finite free position (in additive FFP) if $\chi_{A+B}(x) = \chi_{A}(x) \addconv \chi_{B}(x)$ for every $A \in \mc{A}$ and $B \in \mc{B}$. 
	\end{definition}
	\begin{remarks}\label{rem:additiveFFP}
		\begin{enumerate}
             \item []
             \item As shown in \cite{Marcus21}, the property of being in additive FFP can be expressed in terms of mixed discriminants \cite{Bapat89}.
			\item We notice that $\chi_{A}(x) \addconv \chi_{B}(x)$ only depends on the characteristic polynomials of the matrices and not on the matrices themselves.
            \item This is in general not the case for $\chi_{A+B}(x)$, we can only assert that for any pair of matrices $A,B \in \Cnn$, the two leading monomials of $\chi_{A+B}(x)$ and $\chi_{A}(x) \addconv \chi_{B}(x)$ coincide and are equal to $x^n- \trace(A+B)x^{n-1}$.
            \item
            \label{rem:additiveFFP:conj}
            If $A$ and $B$ are in additive FFP then so are $PAP^{-1}$ and $PBP^{-1}$, for any regular matrix $P$.
		\end{enumerate}
	\end{remarks}
	
	\subsection{Matrices of dimension $2 \times 2$}
	\label{subsec:2x2}
	Let us start with an explicit analysis of $2 \times 2$ matrices. 
	\begin{proposition}
		The matrices $A=(a_{ij}), B=(b_{ij}) \in \mathcal{M}_2$ are in additive FFP if and only if $(a_{11} - a_{22}) (b_{22} - b_{11}) = 2(a_{12} b_{21} + a_{21} b_{12})$.
	\end{proposition}	
	\begin{proof}
		Let
		$$
		A = \begin{bmatrix*}[r]
			a_{11} & a_{12} \\
			a_{21} & a_{22}
		\end{bmatrix*}, \qquad \qquad
		B = \begin{bmatrix*}[r]
			b_{11} & b_{12} \\
			b_{21} & b_{22}
		\end{bmatrix*}.
		$$
		Then
		\begin{align*}
			\chi_{A+B}(x) &= x^2 - (a_{11}+a_{22}+b_{11}+b_{22})x \\ 
			&+ (a_{11}+b_{11})(a_{22}+b_{22}) - (a_{12}+b_{12})(a_{21}+b_{21})	
		\end{align*}
		and
		\begin{align*}
			\chi_{A}(x) &\addconv \chi_{B}(x) = x^2 - (a_{11}+a_{22}+b_{11}+b_{22})x \\
			&+ (a_{11} a_{22} - a_{12} a_{21} + b_{11} b_{22} - b_{12} b_{21}) + \frac{1}{2}(a_{11}+a_{22})(b_{11}+b_{22}).
		\end{align*}
		We see that both $\chi_{A}(x) \addconv \chi_{B}(x)$ and $\chi_{A}(x) \addconv
		\chi_{B}(x)$ are monic polynomials with the same linear
                coefficient given by the negative sum of the traces   (in fact, this holds for any two $n \times n$
                matrices, cf.\ Remark~\ref{rem:additiveFFP}).
                Now FFP only depends on the constant terms and it follows that
                $A$ and $B$ are in additive FFP if and only if
		\begin{equation}
			\label{eq:addFFP2x2}
			(a_{11} - a_{22}) (b_{22} - b_{11}) = 2(a_{12} b_{21} + a_{21} b_{12}).
		\end{equation}
	\end{proof}
	Since the property of being in additive FFP is invariant under conjugating both $A$ and $B$ with the same matrix, we can assume that $A$ is of the form 
	$$
	A = \begin{bmatrix*}[c]
		a_{11} & a_{12} \\
		0 & a_{22}
	\end{bmatrix*}.
	$$
	Then condition \eqref{eq:addFFP2x2} for $A$ and $B$ to be in additive FFP becomes:
	\begin{equation}
		\label{eq:addFFPtriang2x2}
		(a_{11} - a_{22}) (b_{22} - b_{11}) = 2a_{12} b_{21}.
	\end{equation}
	
	By \eqref{eq:addFFPtriang2x2}, $A$ is in additive FFP with itself if and only if
	\begin{equation*}
		a_{11} = a_{22},
	\end{equation*}
	that is, the eigenvalues of $A$ satisfy $\lambda_1 = \lambda_2$, i.e.,
        it is the sum of a nilpotent and a scalar matrix.
        This equivalence is true in general, see Theorem~\ref{thm:selfaddfree}.
	
	When $A$ (or $B$) is a scalar matrix then both sides of
        \eqref{eq:addFFP2x2} vanish,
        so that $A$ and $B$ are in additive FFP. Moreover, when $A$ and $B$ are in additive FFP then it is easy to see that equality \eqref{eq:addFFP2x2} holds also for $\lambda_1 A + \lambda_2 I$ and $\mu_1B + \mu_2 I$, for every $\lambda_1, \lambda_2,
    \mu_1, \mu_2 \in \CC$.
    In fact, we have the following more general result. We will see later that it does not hold in higher dimensions, see Remark~\ref{rem:FFP:preserve}.
	\begin{proposition}
		\label{pr:poly2x2}
		Let  $A=(a_{ij}), B=(b_{ij}) \in \mathcal{M}_2$ be in additive FFP. Then $p(A)$ and $q(B)$ are in additive FFP, for any polynomials $p(x)$ and $q(x)$. Moreover, if $A$ is regular then $A^{-1}$ and $q(B)$ are in additive FFP.
	\end{proposition}
	
	\begin{proof}
          We infer from the Cayley--Hamilton theorem that every holomorphic
          function $f(A)$ can be represented as a linear function of the form
          $a_0I + a_1A$
          and the claim follows from Proposition~\ref{pr:closedness}.
	\end{proof}

	\subsection{Matrices of dimension $3 \times 3$}
	\label{subsec:3x3}
	Let us look at matrices of dimension $3 \times 3$ that are in additive
        FFP. As in the case of matrices of dimension $2 \times 2$, the  two
        leading coefficients  always coincide among the two polynomials,
        so that we obtain two equations for the remaining coefficients.
        In order to avoid equations with many variables, we restrict the
        discussion to diagonal matrices, which reduces the number of variables
        to  six.
	Let
	$$
	A = \begin{bmatrix*}[r]
		\lambda_1 & 0 & 0 \\
		0 &\lambda_2 & 0 \\
		0 & 0 & \lambda_3
	\end{bmatrix*}, \qquad \qquad
	B = \begin{bmatrix*}[r]
		\mu_1 & 0 & 0 \\
		0 &\mu_2 & 0 \\
		0 & 0 & \mu_3
	\end{bmatrix*}.
	$$
	Then
	\begin{align*}
		\chi_{A+B}(x) &= x^3 - ((\lambda_1+\lambda_2+\lambda_3)+(\mu_1+\mu_2+\mu_3)) x^2 \\ 
		&+ ((\lambda_1+\mu_1)(\lambda_2+\mu_2)+(\lambda_1+\mu_1)(\lambda_3+\mu_3)+(\lambda_2+\mu_2)(\lambda_3+\mu_3)) x \\
		&- (\lambda_1+\mu_1)(\lambda_2+\mu_2)(\lambda_3+\mu_3)
	\end{align*}
	and
	\begin{align*}
		\chi_{A}(x) &\addconv \chi_{B}(x) = x^3 - ((\lambda_1+\lambda_2+\lambda_3)+(\mu_1+\mu_2+\mu_3)) x^2 \\
		&+ ((\lambda_1 \lambda_2+\lambda_1 \lambda_3+\lambda_2 \lambda_3) +\frac{2}{3}(\lambda_1+\lambda_2+\lambda_3)(\mu_1+\mu_2+\mu_3) \\
		&+ (\mu_1 \mu_2+\mu_1 \mu_3+\mu_2 \mu_3)) x \\
		&-(\lambda_1 \lambda_2 \lambda_3 + \frac{1}{3}(\lambda_1+\lambda_2+\lambda_3)(\mu_1\mu_2+\mu_1\mu_3+\mu_2\mu_3) \\
		&+ \frac{1}{3}(\lambda_1\lambda_2+\lambda_1\lambda_3+\lambda_2z_3)
		(\mu_1+\mu_2+\mu_3) +\mu_1\mu_2\mu_3).
	\end{align*}
	By comparing the coefficients of $x$ and the free terms of both polynomials we get that the diagonal matrices $A$ and $B$ are in additive FFP if and only if the following two equations are satisfied:
	\begin{equation}
		\label{eq:addFFP3x3a}
		\lambda_1\mu_2+\lambda_1\mu_3+\lambda_2\mu_1+\lambda_2\mu_3+\lambda_3\mu_1+\lambda_3\mu_2 = 2(\lambda_1\mu_1+\lambda_2\mu_2+\lambda_3\mu_3)
	\end{equation}
	and
	\begin{align}
		\label{eq:addFFP3x3b}
		&\lambda_1\mu_1\mu_2+\lambda_1\mu_1\mu_3+\lambda_2\mu_1\mu_2+\lambda_2\mu_2\mu_3+\lambda_3\mu_1\mu_3+
		\lambda_3\mu_2\mu_3 \nonumber \\
		&+ \lambda_1\lambda_2\mu_1+\lambda_1\lambda_2\mu_2+\lambda_1\lambda_3\mu_1+\lambda_1\lambda_3\mu_3+\lambda_2\lambda_3\mu_2+\lambda_2\lambda_3\mu_3 \\
		&= 2(\lambda_1\mu_2\mu_3+\lambda_2\mu_1\mu_3+\lambda_3\mu_1\mu_2+\lambda_1\lambda_2\mu_3+\lambda_1\lambda_3\mu_2+\lambda_2\lambda_3\mu_1) 
		\nonumber.
	\end{align}

	\section{The lattice of additive finite free affine algebraic sets of matrices}
	\label{sec:addlattice}
        In this section we extend the discussion from single matrices to sets
        of matrices in FFP.
        To this end, we consider the set $\mc{V} \subseteq \Cnn$ of all matrices
        that are in additive FFP with every matrix $A$ in a given set
        $\mc{A}$. This set $\mc{V}$ forms a (non-irreducible) affine algebraic set.

	For general $n \times n$ matrices $A=(A_{ij})$ and $B=(B_{ij})$, let 
	$$
	\chi_{A+B}(x) = \sum_{k=0}^{n} c_k(A_{11},A_{12},\ldots,A_{nn},B_{11},B_{12},\ldots,B_{nn}) x^{n-k}
	$$
	and
	$$
	\chi_{A}(x) \addconv \chi_{B}(x) = \sum_{k=0}^{n} d_k(A_{11},A_{12},\ldots,A_{nn},B_{11},B_{12},\ldots,B_{nn}) x^{n-k},
	$$
	where $c_k(A_{11}, \ldots,A_{nn},B_{11},\ldots,B_{nn})$, $d_k(A_{11},\ldots,A_{nn},B_{11},\ldots,B_{nn})$
	are multilinear homogeneous polynomials in the variables $A_{11},\ldots,A_{nn}$, $B_{11},\ldots,B_{nn}$ of degree $k$.
	By definition \ref{eq:addconv} of the additive convolution, it is clear that a coefficient of $x^{n-k}$ in $\chi_{A}(x) \addconv \chi_{B}(x)$ has as summands the coefficient $a_k$ of $\chi_{A}(x)$, as well as the coefficient $b_k$ of $\chi_{B}(x)$. These are also summands of $\chi_{A+B}(x)$.
	Moreover, the coefficients of $x^n$ and of $x^{n-1}$ (1 and $-\trace(A+B)$, respectively) are the same in $\chi_{A}(x) \addconv \chi_{B}(x)$ and in $\chi_{A+B}(x)$.
	That is, two specific matrices $A=(a_{ij}), B=(b_{ij}) \in \Cnn$ are in additive FFP if and only if
	$$
	c_k(a_{11},\ldots,a_{nn},b_{11},\ldots,b_{nn}) = d_k(a_{11},\ldots,a_{nn},b_{11},\ldots,b_{nn}),
	$$
	for $k=2,\ldots,n$.
	
	Suppose now that $A=(a_{ij}) \in \Cnn$ is a specific matrix.
	Then the matrices $B$ that are in additive FFP with $A$ are the zero locus of
	the following set of polynomials over $\CC$:
	\begin{align*}
		p_{A,k}&(B_{11},\ldots,B_{nn}) := \\
		&c_k(a_{11},\ldots,a_{nn},B_{11},\ldots,B_{nn}) - d_k(a_{11},\ldots,a_{nn},B_{11},\ldots,B_{nn}),
	\end{align*}
	for $k=2,\ldots,n$.
	We associate with $A$ the ideal $I_A$ generated by these polynomials:
	$$
	I_A := \langle \, p_{A,2}(B_{11},\ldots,B_{nn}), \ldots, p_{A,n}(B_{11},\ldots,B_{nn})\, \rangle.
	$$
	The (non-irreducible) affine algebraic set defined by $I_A$ is
	\begin{align*}
		\mc{V}(I_A) := \{&B=(b_{ij}) \in \Cnn \, : \\ &p_{A,2}(b_{11},\ldots,b_{nn})=\cdots=p_{A,n}(b_{11},\ldots,b_{nn})=0\},
	\end{align*}
	and this is the set of matrices $B \in \Cnn$ that are in additive FFP with $A$:
	$$\mc{V}(I_A) = \{ B \in \Cnn \, : \, \chi_{A+B}(x) = \chi_{A}(x) \addconv \chi_{B}(x) \}.$$
	We denote by ${\{A\}}^{\addconv}$ the affine algebraic set $\mc{V}(I_A)$. More generally, we define the following.
	\begin{definition}
		Given a set $\mc{A} \subseteq \Cnn$, the \textbf{additive finite free complement} of $\mc{A}$ (or the \textbf{additive finite free affine algebraic set} $\act{\mc{A}}$)
		is the set
		$$
		\act{\mc{A}} := \mc{V}(I_{\mc{A}}) = \{ B \in \Cnn \, : \, 
		B \in \mc{V}(I_A),
		\mbox{ for all } A \in \mc{A} \}
                =\bigcap_{A\in \mc{A}} \mc{V}(I_A)
                ,
		$$
		which is the affine algebraic set of matrices defined by the ideal $I_{\mc{A}}$ that is generated by the set of polynomials $p_{A,k}(B_{11},\ldots,B_{nn})$, for $A \in \mc{A}$ and $k=2,\ldots,n$.
	\end{definition} 
	\begin{definition}
		When $\act{\mc{V}} = \mc{W}$ and $\act{\mc{W}} = \mc{V}$, we say that the affine algebraic sets $\mc{V},\mc{W}$ form an \textbf{additive finite free complementary pair} (an \textbf{additive complementary pair}, for short). 
	\end{definition}
	\noindent Note that when $\mc{W} = \act{\mc{V}}$ then $\mc{W} = \biact{\mc{W}}$.
	
	We show next that every additive finite free (non-irreducible) affine algebraic set of matrices contains the set of scalar matrices. 
	\begin{proposition}
		\label{pr:closedness}
		If $A$ and $B$ are in additive FFP then so are $A + \lambda I$ and $B$, for every $\lambda \in \CC$.
	\end{proposition}
	\begin{proof}
		Let $A,B \in \Cnn$.
		Then $\chi_{A+\lambda I}(x) = \det(xI -(A + \lambda I) = \det((x -\lambda) I-A) = \chi_{A}(x-\lambda)$. Similarly, $\chi_{A+\lambda I + B}(x) = \chi_{A+B}(x-\lambda)$. 
		By assumption,
		$\chi_{A}(x)	\addconv \chi_{B}(x) =\chi_{A+B}(x)$ and 
		it follows that
		\begin{align*}
			\chi_{A+\lambda I}(x) &\addconv \chi_{B}(x) = 
			\frac{1}{n!} \sum_{k=0}^{n} \chi_{A+\lambda I}^{(k)}(x) \chi_{B}^{(n-k)}(0) \\
			&=
			\frac{1}{n!} \sum_{k=0}^{n} \chi_{A}^{(k)}(x-\lambda) \chi_{B}^{(n-k)}(0) \\
			&= \chi_{A+B}(x-\lambda) = \chi_{A+\lambda I + B}(x).
		\end{align*}
	\end{proof}
	
	\begin{remark}
          \label{rem:FFP:preserve}
          Proposition~\ref{pr:poly2x2} raises the question whether other algebraic
          operations preserve additive FFP also in higher dimensions.
          The following example shows that the obvious candidates fail:
		The matrices
		$$
		A =
		\begin{bmatrix}
			1 & 0 & 0 \\
			0 & 2 & 0 \\
			0 & 0 & 3 
		\end{bmatrix}, \qquad \qquad
		B =
		\begin{bmatrix}
			1 & -1 & 0 \\
			-1 & 13 & -3 \\
			0 & -3 & 1 
		\end{bmatrix}
		$$
		are in additive FFP but the following pairs are not:
		\begin{enumerate}[(i)]
			\item   $(\lambda A,B)$,  unless   $\lambda\in\{0,1\}$;
			\item  $(A^2,B)$;
			\item  $(A^{-1},B)$.
		\end{enumerate}
		
	\end{remark}
	
	\begin{definition}
		\label{def:addvariety}
		Let $\mc{A} \subseteq \Cnn$ and let $\mc{V} = \biact{\mc{A}}$. Then we say that $\mc{A}$ is a \textbf{generating set} of $\mc{V}$, written $\mc{V} = \langle
		\mc{A} \rangle$.
		We define the \textbf{additive finite free rank} of $\mc{V}$, denoted $\mathrm{ark}(\mc{V})$, to be the minimal cardinality of a generating set of $\mc{V}$.
	\end{definition}
	\begin{proposition}
		Every additive finite free affine algebraic set $\mc{V}\subseteq \Cnn$ is of finite additive finite free rank, with $\displaystyle \mathrm{ark}(\mc{V}) \leq  \sum_{k=1}^{n-1} k! {n \choose k} ^2$.
		\label{pr:rankbound}
	\end{proposition}
	\begin{proof} Let $\mc{A}$ be a generating set of $\mc{V}$, that is, $\mc{V} = \biact{\mc{A}}$, and let $I_{\mc{A}}$ be the corresponding ideal. That is, $I_{\mc{A}}$ is the sum of the ideals $I_A$, for $A \in  \mc{A}$, where $I_{\mc{A}}$ is generated by the finitely many polynomials $p_{A,k}(B_{11},\ldots,B_{nn})$. Since $I_{\mc{A}}$ is an ideal in a Noetherian ring, it is finitely generated. It follows that there exists a finite set $\mc{A'} \subseteq \mc{A}$, such that $I_{\mc{A}}=I_{\mc{A'}}$. But then $\mc{V} = \biact{\mc{A}} = \biact{\mc{A'}}$, that is $\mc{V} = \langle \mc{A'} \rangle$.
		
		Each polynomial $p_{A,k}(B_{11},\ldots,B_{nn})$, for $k=2,\ldots,n$, is of degree at most $k-1$, since the part that is of degree $k$ does not involve the elements of $A$ and is the same in $c_k$ and $d_k$, and therefore vanishes in $p_{A,k}$. Similarly, the part of degree 0 comes just from $A$ and also vanishes in $p_{A,k}$. 
		The monomials in $p_{A,k}(B_{11},\ldots,B_{nn})$ are then of degree $1 \le d \le n-1$ and of the form $B_{i_1 j_1}B_{i_2 j_2}\cdots B_{i_l j_l}$, where $\{i_1, \ldots, i_l\}$ and $\{j_1, \ldots, j_l\}$ are subsets of size $l$ of $\{1, \ldots, n\}$ (by the way the determinant is defined). It follows that the generators of the ideal $I_{\mc{A}}$ span linearly a subspace of a vector space of dimension
		$$
		\sum_{k=1}^{n-1} {n \choose k} \frac{n!}{(n-k)!} = \sum_{k=1}^{n-1} k! {n \choose k} ^2,
		$$
		which gives an upper bound on the number of generators of $I_{\mc{A}}$ and, thus, on $\mathrm{ark}(\mc{V})$.
	\end{proof}
	\begin{remark}
		The examples given in this paper suggest that the affine algebraic sets are probably of additive finite free rank much smaller than the bound given in Proposition \ref{pr:rankbound}, more likely of order $\mc{O}(n^2)$, e.g., the Krull dimension of the coordinate ring $\CC[B_{11},\ldots,B_{nn}]/I_A$, possibly even that 
		$\mathrm{ark}(\mc{V}) + \mathrm{ark}(\mc{W})) = n^2-1$, for affine algebraic sets $\mc{V} \subseteq \Cnn$ and $\mc{W} \subseteq \Cnn$ that form an additive complementary pair. 
	\end{remark}
	The collection of additive finite free (non-irreducible) affine algebraic sets forms a \textbf{lattice}
	$(\LL_n, \vee, \wedge, \Sn{I}, \Cnn)$
	under set inclusion, where: \\
	\begin{itemize}
		\item The set $\PBDD$ of scalar matrices is the bottom element $\bot$ of $\LL$ and the set $\Cnn = \CC^{n \times n}$ is the top element $\top$.
		\item For every $\mc{V}, \mc{W}$ in $\LL_n$, their \textbf{join} $\mc{V} \vee \mc{W} := \langle \mc{V} \cup \mc{W} \rangle$ and their \textbf{meet} $\mc{V} \wedge \mc{W} := \mc{V} \cap \mc{W}$ are in $\LL_n$. 
	\end{itemize} 
	Some of the properties of the lattice $\LL_n$ that are easy to verify are:
	\begin{enumerate}
		\item $\LL_n$ is associative, commutative and distributive with respect to $\vee$ and $\wedge$.
		\item $\mc{V} \wedge \bot = \bot$ and $\mc{V} \vee \bot = \mc{V}$, for every $\mc{V} \in \LL_n$.
		\item $\mc{V} \wedge \top = \mc{V}$ and $\mc{V} \vee \top = \top$,	for every $\mc{V} \in \LL_n$.
		\item $(\mc{V} \vee \mc{W})^{\,\addconv} = \act{\mc{V}} \wedge \act{\mc{W}}$ and $(\mc{V} \wedge \mc{W})^{\,\addconv} = \act{\mc{V}} \vee \act{\mc{W}}$, for every $\mc{V}, \mc{W} \in \LL_n$ (De Morgan laws).
	\end{enumerate}
	\begin{remarks}
		\begin{enumerate}
			\item []
			\item The complement of an affine algebraic set in $\LL_n$ does not obey the standard definition of a complement in a lattice, that is, it does not hold in general that $\mc{V} \vee \act{\mc{V}} = \top$ and $\mc{V} \wedge \act{\mc{V}} = \bot$.
			\item Since every additive finite free affine algebraic set contains the subset of scalar matrices, we could have defined the lattice in the quotient space $\MM_n := \Cnn / \PBDD$. 
		\end{enumerate}
	\end{remarks}
	Let us look at the special case of $2 \times 2$ matrices.
	\begin{proposition}
		Let the affine algebraic sets $\mc{V}, \mc{W} \subseteq \mathcal{M}_2$ form an additive complementary pair. Then
		\begin{enumerate}
			\item $\mc{V}, \mc{W}$ are subspaces of $\mathcal{M}_2$;
			\item $\mathrm{ark}(\mc{V}) + \mathrm{ark}(\mc{W}) = 3$.
		\end{enumerate}
	\end{proposition}
	\begin{proof}
		\begin{enumerate}
			\item By \eqref{eq:addFFP2x2}, the elements of the matrices $A \in \mc{V}$, that are in additive FFP with a matrix $B=(b_{ij}) \in \mc{W}$, are the solution set of the linear homogeneous equation
			\begin{equation}
				\label{eq:lineq2x2}
				(b_{22} - b_{11})A_{11} - (b_{22} - b_{11})A_{22} -2_{21}A_{12}  - 2b_{12}A_{21} = 0
			\end{equation}
			and thus form a subspace of $\mathcal{M}_2$. It follows that $\mc{V}$, and similarly $\mc{W}$, are subspaces of $\mathcal{M}_2$.
			\item For each matrix $B \in \mc{W}$, we obtain a linear equation of the form \eqref{eq:lineq2x2}. Hence, it is sufficient to find the matrices $A$ that are in additive FFP with respect to a linear basis of the subspace $\mc{W}$. One of the elements of $\mc{W}$ (written as a row vector) is $(b_{11}, b_{22}, b_{12}, b_{21}) = (1,1,0,0)$, representing the identity matrix. Thus, by reduction through this row, we may assume that the other basis elements have zero in their first entry. When $(b_{11}, b_{22}, b_{12}, b_{21}) = (1,1,0,0)$ then Equation \eqref{eq:lineq2x2} vanishes, and so, we get that $\mc{V}$ is the set of solutions of an homogeneous system of linear equations of the form
			\begin{equation}
				\label{eq:sys2x2}
				c_{1}A_{11} - c_{1}A_{22} + c_{2}A_{12} + c_{3}A_{21} = 0.
			\end{equation}
			The number of these equations is $\dim (\mc{W})-1$ (since the identity matrix does not contribute an equation). Moreover, since these equations are independent, we cannot take a smaller number of equations in order to define $\mc{V}$. A basis for $\mc{V}$ (a basis for the solution of the system of equations) consists of $4 - (\dim (\mc{W})-1) = 5 - \dim (\mc{W})$ elements. It follows from the above discussion that
			$\mathrm{a rk}(\mc{V}) = \dim (\mc{V}) - 1$,
			$\mathrm{ark}(\mc{W}) = \dim (\mc{W}) - 1$ (after omitting the identity matrix), and 
			\begin{equation*}
				\label{eq:rank2x2}
				\mathrm{ark}(\mc{V}) + \mathrm{ark}(\mc{W}) = 5 - 2 = 3.
			\end{equation*}
		\end{enumerate}
	\end{proof}
	
\section{Additive finite free complementary pairs}
\label{sec:addcomp}

In this section we present pairs of (non-irreducible) affine algebraic sets
$(\mc{V}, \mc{W}) \subseteq \Cnn \times \Cnn$ that form additive complementary
pairs. These are: $(\DD,\PB)$, the set of diagonal matrices and the set of
principally balanced matrices; $(\PBUT,\UT)$ (resp.\ $(\PBLT,\LT)$), the set of
upper (resp.\ lower)
triangular matrices with constant diagonal and the set of upper (resp.\ lower)
triangular matrices; $(\Sn{S},\Cnn)$, the set of scalar matrices and the set of
all $n \times n$ matrices.
	
	We will make frequent use of the following lemma. We denote by $E_{kl} \in \Cnn$ the matrix with 1 at position $(k,l)$ and 0 elsewhere.
	\begin{lemma}
		\label{lem:Ekl}
		Let $A=(a_{ij}) \in \Cnn$ be a matrix having an off-diagonal non-zero entry $a_{lk}$. Then $\chi_{A+E_{kl}}(x) \neq \chi_{A}(x) = \chi_{A}(x) \addconv \chi_{E_{kl}}(x)$.
	\end{lemma}
	\begin{proof}
		Let us look at the coefficient of $x^{n-2}$ in $\chi_{A+E_{kl}}(x)$ and in $\chi_{A}(x)$. The only term that is affected by the addition of $E_{kl}$ to $A$ is the principal minor with indices $k,l$. The difference is
		$$
		(a_{kk}a_{ll}-(a_{kl}+1)a_{lk}) -
		(a_{kk}a_{ll}-a_{kl}a_{lk}) = -a_{lk}\neq 0.
		$$
		It follows that $\chi_{A+E_{kl}}(x) \neq \chi_{A}(x)$.
		On the other hand, $\chi_{E_{kl}}(x) = x^n$ and by Example~\ref{ex:unit}, $\chi_{A}(x) \addconv \chi_{E_{kl}}(x) = \chi_{A}(x)$.
	\end{proof}
	
	\subsection{Diagonal and principally balanced matrices}
        Consider first diagonalizable matrices.
        Because of invariance under conjugation
        (Remark~\ref{rem:additiveFFP} (4)), we may assume that one of the matrices is diagonal.
        It turns out that the following concept is crucial.
	\begin{definition}
		A matrix $B \in \Cnn$ is called \textbf{principally balanced} (or has the \textbf{symmetrized principal minors} property \cite{HO17}) if, for every $i$, $1 \le i \le n$, the values of all principal minors of $B$ of order $i$ coincide.
		We denote by $\PB \subseteq \Cnn$ the family of $n \times n$ matrices which are principally balanced.
        \end{definition}
        This is a special case of the principal minor assignment problem
        \cite{GriffinTsatsomeros:2006:principal2,RisingKuleszaTaskar:2015:efficient}
        going back to Stouffer \cite{Stouffer:1924:independence}.
        
	It it easy to see that the family of principally balanced matrices is
        invariant under conjugation by permutation matrices and diagonal matrices.
	Clearly, all $1 \times 1$ matrices are principally balanced, as well as all $2 \times 2$ matrices with constant diagonal and all $n \times n$ triangular matrices with constant diagonal.
	Here are examples of specific principally balanced matrices.
	\begin{example}
		The following $3 \times 3$ matrix is principally balanced: all principal minors of order $1$ (the diagonal entries) equal $1$ and all principal minors of order $2$ equal $-11$:
		$$
		\begin{bmatrix*}[r]
			1 & 2 & 3 \\
			6 & 1 & -12 \\
			4 & -1 & 1	
		\end{bmatrix*}.
		$$
	\end{example}
	
	\begin{example}
		Let $A$ be the $n \times n$ matrix with entries $a_{ij}=\frac{i}{j}$, $i=1,\ldots,n$:
		$$
		A =
		\begin{bmatrix*}[r]
			1 & \frac{1}{2} & \frac{1}{3} \cdots & \frac{1}{n}\\[1mm]
			2 & 1 & \frac{2}{3} \cdots & \frac{2}{n}\\[1mm]
			3 & \frac{3}{2} & 1 \cdots & \frac{3}{n}\\[-1mm]
			\vdots & &  \ddots & \vdots \\[-1mm]
			n & \frac{n}{2} & \frac{n}{3} \cdots & 1	
		\end{bmatrix*}.
		$$
		The rank of $A$ is 1, so all minors (not just principal) of order $k \geq 2$ equal 0, while the principal minors of order 1 equal 1. Thus, $A$ is principally balanced.
		
		More generally, let $A_i$ denote the $i$-th row of $A$, $i=1,\ldots,n$. When
		all entries of $A_1$ are non-zero and $A_i = \frac{a_{11}}{a_{1i}}A_1$,
		$i=1,\ldots,n$, then $A$ is principally balanced with all principal minors of
		order $1$ equal $a_{11}$ and all minors of order $k \geq 2$ equal $0$. In fact,
		it is easy to show that for any positive integer $k$, $A^k=(a_{11}n)^{k-1} A$
		and so, for any polynomial $p(x)$ with $p(0)=0$, $p(A)$ is a scalar multiple of $A$ and thus of the same form as $A$ (just that the principal minors of order 1 are multiplied by $\lambda$) and, in particular, is principally balanced. Moreover, $p(A)$ is principally balanced for every polynomial $P$ since the set of principally balanced matrices is closed under addition of a scalar matrix by Proposition~\ref{pr:closedness} and Theorem~\ref{thm:addcomp_1}.
	\end{example}
	\begin{lemma}
		\label{lem:prbalanced}
		Let $B \in \PB$ be principally balanced and let $m_i$, for $i=0,\ldots,n$, be its $i \times i$ principal minor, where $m_0$ is defined to be $1$.
		Then $\chi_B (x) = \sum_{i=0}^{n} (-1)^i {n \choose i} m_i x^{n-i}$.
	\end{lemma}
	\begin{proof}
		For each $i$, the coefficient of $x^{n-i}$ in $\chi_B (x)$ is $(-1)^i$ times the sum of all principal minors of order $i$.
		There are ${n \choose i}$ such principal minors and each one equals $m_i$. 	
	\end{proof}

        \begin{remark}
          The family of principally balanced matrices can be equivalently defined by the
          property of having equal ``cycle sums'' of the same order:
          $c_I = c_J$, for all $I,J \subseteq [n]$ with $\vert I \vert = \vert J
          \vert$,
          where $$c_I=\sum_{\substack{I=\{i_1,\ldots,i_k\} \\ i_1=\min I}}
          a_{i_1, i_2}\cdots  a_{i_{k-1}, i_k} a_{i_k, i_1},$$
          are the cycle sums of a matrix $A=(a_{ij})$,
          see \cite{HO17}.
          Indeed the authors prove that principal minors
          and cycle sums are related by M\"{o}bius inversion on the lattice of set
          partitions.

          This can be interpreted as follows.
          Let $B\in \PB$ be a principally balanced matrix with principal minors
          $m_k$
          and cycle sums $c_k$.
          Consider $\tilde{m}_k = (-1)^k m_k$ as moments of a formal random
          variable (or ``umbra'' in the sense of Rota \cite
          {RotaShen:2000:combinatorics}) $Z$,
          then
          $$
          \chi_B(x) = \EE(x+Z)^n 
          $$
          and  the corresponding classical
          cumulants of $Z$ are the numbers $\tilde{c}_k = (-1)^k c_k$.
        \end{remark}

	We denote by $[n]$ the set $\{1,\ldots,n\}$.
	When $M \in \Cnn$ and $S \subseteq [n]$ with $|S|=k$, we denote by $M_S$ the $k \times k$ principal submatrix of $M$ obtained after removing the rows and columns with indices in $[n]-S$.
	\begin{theorem}
		\label{thm:addcomp_1}
		Let $\DD \subseteq \Cnn$ be the family of $n \times n$ diagonal matrices and let $\PB \subseteq \Cnn$ be the family of $n \times n$ principally balanced matrices.
		Then $\DD$ and $\PB$
		form an additive complementary pair.
	\end{theorem}
	\begin{proof}
		The theorem clearly holds for $n=1$, so let us assume that $n \geq 2$.
		First, we show that for arbitrary $D \in \DD$ and $B \in \PB$, the equality $\chi_{D+B}(x) = \chi_{D}(x) \addconv \chi_{B}(x)$ holds.
		Let $D = \mathrm{diag} (d_1,\ldots,d_n)$, hence
		\begin{align}
			\label{eq:chi_D}
			\chi_D (x) &= \sum_{i=0}^{n} c_i x^{n-i} =\prod_{i=1}^n (x-d_i) \nonumber \\
			&= \sum_{i=0}^{n} \bigg( (-1)^i \sum_{J \subseteq [n],\, |J|=i} \det(D_J) \bigg) x^{n-i} \nonumber \\
			&= \sum_{i=0}^{n} \bigg( (-1)^i \sum_{J \subseteq [n],\, |J|=i} d_J \bigg) x^{n-i},
		\end{align}
		where $d_J = d_{j_1} \cdots d_{j_i}$ for $J=\{j_1, \ldots, j_i\}$.
		
		For $B \in \PB$, we have
		\begin{equation}
			\label{eq:chi_B}
			\chi_B (x) = \sum_{i=0}^{n} b_i x^{n-i} = \sum_{i=0}^{n} (-1)^i {n \choose i} m_i x^{n-i},
		\end{equation}
		where $m_i$, $i=0,\ldots,n$, is the value of a principal minor of $B$ of order $i$.
		
		The characteristic polynomial of $B+D$ is
		\begin{align}
			\label{eq:chi_D+B_1}
			\chi_{D+B}(x)
			&= \sum_{i=0}^{n} \bigg( (-1)^i \sum_{I \subseteq [n],\, |I|=i} \det((D+B)_I) \bigg) x^{n-i} \nonumber \\
			&= \sum_{i=0}^{n} \bigg( (-1)^i \sum_{I \subseteq [n],\, |I|=i} \sum_{j=0}^{i} \sum_{J \subseteq I,\, |J|=j}m_{i-j}d_J \bigg) x^{n-i}.
		\end{align}
		We observe that when we compute $\det((D+B)_I)$, with $|I|=i$, in \eqref{eq:chi_D+B_1}, then we sum over the products of $d_J$, which is a principal minor of order $j \leq i$ of $D$, and a principal minor of order $i-j$ of $B$.
		But all principal minors of order $i-j$ of $B$ equal $m_{i-j}$.
		It follows that for each $J$, the term $m_{i-j}d_J$ appears $n-j \choose i-j$ times in \eqref{eq:chi_D+B_1}, the number of principal minors of order $i-j$ of $B$ with indices of the submatrices outside of $J$, while the submatrix $D_J$ of $D$ is kept fixed.
		Hence
		\begin{align}
			\label{eq:chi_D+B_2}
			\chi_{D+B}(x)
			&= \sum_{i=0}^{n} \bigg( (-1)^i \sum_{j=0}^{i} \sum_{J \subseteq [n],\, |J|=j} {n-j \choose i-j} m_{i-j}d_J \bigg) x^{n-i} \nonumber \\
			&= \sum_{i=0}^{n} \bigg( (-1)^i \sum_{j=0}^{i}  {n-j \choose i-j} m_{i-j} \sum_{J \subseteq [n],\, |J|=j}d_J \bigg) x^{n-i}.
		\end{align}
		
		Next, we compute the additive convolution of $\chi_{D}(x)$ and $\chi_{B}(x)$.
		By \eqref{eq:addconv},
		\begin{align}
			\chi_{D}(x) \addconv \chi_{B}(x) &= \sum_{i=0}^{n} f_i x^{n-i}
			= \sum_{i=0}^{n} c_i x^{n-i} \addconv \sum_{i=0}^{n} b_i x^{n-i}\\
			&= \sum_{i=0}^{n} \bigg( \sum_{j=0}^{i} \frac{(n-j)!(n-i+j)!}{n!(n-i)!} c_j b_{i-j} \bigg) x^{n-i}.
		\end{align}
		Substituting for $c_j$ the appropriate coefficient from \eqref{eq:chi_D} and for $b_{i-j}$ the appropriate coefficient from \eqref{eq:chi_B}, we get
		\begin{align}
			\label{eq:chiD_conv_chiB}
			f_i &= \sum_{j=0}^{i} \frac{(n-j)!(n-i+j)!}{n!(n-i)!} (-1)^j \sum_{J \subseteq [n],\, |J|=j} d_J (-1)^{i-j} {n \choose i-j} m_{i-j} \nonumber \\
			&= (-1)^i \sum_{j=0}^{i} \frac{(n-j)!(n-i+j)!n!}{n!(n-i)!(n-i+j)!(i-j)!} m_{i-j} \sum_{J \subseteq [n],\, |J|=j} d_J \nonumber \\
			&= (-1)^i \sum_{j=0}^{i}  {n-j \choose i-j} m_{i-j} \sum_{J \subseteq [n],\, |J|=j}d_J.
		\end{align}
		We got in \eqref{eq:chiD_conv_chiB} the same coefficient of $x^{n-i}$ as in  \eqref{eq:chi_D+B_2}, hence $\chi_{D+B}(x) = \chi_{D}(x) \addconv \chi_{B}(x)$.
		It follows that $\PB \subseteq \act{\mc{D}}_n$ and $\DD \subseteq \act{\mc{B}}_n$.
		
		We show now that $\act{\mc{D}}_n \subseteq \PB$.
		Let $B=(b_{ij}) \in \act{\mc{D}}_n$.
		Let $K \subseteq [n]$, $|K|=k$, $0 < k < n$, and let $D(\lambda, K)=(d_{ij})$ be the parametric diagonal matrix with entries
		$$
		d_{ij} = \left\{ 
		\begin{array}{rl}
			\lambda \quad &\mbox{for } i=j \in K, \\
			0 \quad &\mbox{otherwise.}
		\end{array}
		\right.
		$$
		Evaluating $\chi_{D(\lambda,K)+B}(x)$ at $x=0$ gives
		$$
		\chi_{D(\lambda,K)+B}(0) = (-1)^n \det (D(\lambda,K)+B) = (-1)^n \det(B_{[n]-K})\lambda^{k} + q(\lambda), 
		$$
		where $q(\lambda)$ is a polynomial in $\lambda$ of degree less than $k$.
		It follows that
		\begin{equation}
			\label{eq:limsum}
			\lim_{\lambda \to \infty} \frac{\chi_{D(\lambda,K)+B}(0)}{\lambda^{k}} = (-1)^n \det(B_{[n]-K}).
		\end{equation}
		
		Let us now compute this limit for $\chi_{D(\lambda,K)}(x) \addconv \chi_{B}(x)$.
		The characteristic polynomial of $D(\lambda,K)$ is
		$$
		\chi_{D(\lambda,K)}(x)  = \sum_{i=0}^{n} a_i x^{n-i}
		= x^{n-k} (x-\lambda)^{k} = \sum_{i=0}^{n} {k \choose i} (-\lambda)^i x^{n-i},
		$$
		where ${k \choose i}=0$ when $i>k$.
		That is,
		$$
		a_i = \left\{ 
		\begin{array}{ccl}
			(-1)^i {k \choose i}  \lambda^i \quad &\mbox{for}& 0 \leq i \leq k, \\
			0 \quad &\mbox{for}& i > k.
		\end{array}
		\right.
		$$
		The characteristic polynomial of $B$ is
		$$
		\chi_B (x) = \sum_{i=0}^{n} b_i x^{n-i} = \sum_{i=0}^{n} \bigg( (-1)^i \sum_{I \subseteq [n],\, |I|=i} \det(B_I) \bigg) x^{n-i}.
		$$
		That is, $(-1)^i b_i$ is the sum of the ${n \choose i}$ principal minors of order $i$.
		
		By \eqref{eq:addconv},
		\begin{align*}
			(\chi_{D(\lambda,K)} \addconv \chi_{B})(0)
			&= \sum_{i=0}^n \frac{(n-i)!i!}{n!} a_i b_{n-i} \\ 
			&=  \sum_{i=0}^n \bigg( \frac{(-1)^i {k \choose i} (-1)^{n-i}}{{n \choose i}} \sum_{I \subseteq [n],\, |I|=n-i} \det(B_I) \bigg) \lambda^i,
		\end{align*}
		a polynomial of degree at most $k$ in $\lambda$.
		It follows that
		$$
		(\chi_{D(\lambda,K)} \addconv \chi_{B})(0) =
		\bigg( \frac{(-1)^n}{{n \choose k}} \sum_{I \subseteq [n],\, |I|=n-k} \det(B_I) \bigg) \lambda^k + p(\lambda),
		$$
		where $p(\lambda)$ is a polynomial in $\lambda$ of degree less than $k$.
		Then
		\begin{equation}
			\label{eq:limaconv}
			\lim_{\lambda \to \infty} \frac{(\chi_{D(\lambda,K)} \addconv \chi_{B})(0)}{\lambda^{k}} = \frac{(-1)^n}{{n \choose k}} \sum_{I \subseteq [n],\, |I|=n-k}\det(B_I).
		\end{equation}
		Comparing \eqref{eq:limsum} with \eqref{eq:limaconv}, we get that
		$$
		\det(B_{[n]-K}) = \frac{1}{{n \choose k}} \sum_{I \subseteq [n],\, |I|=n-k}\det(B_I),
		$$
		that is, for every $1 \le k \le n$, every principal minor of $B$ of order $n-k$ equals the mean value of the principal minors of order $n-k$. In other words, all principal minors of the same order are equal.
		It follows that $B$ is principally balanced and $\act{\mc{D}}_n \subseteq \PB$.
		Together with the inclusion $\PB \subseteq \act{\mc{D}}_n$ we have
		$$
		\act{\mc{D}}_n = \PB.
		$$
		
		It remains to show that the additive free complement of $\PB$ is $\DD$ and not a larger family.
		Suppose that $A$ has an off-diagonal entry $a_{kl}\ne0$, $k\ne
                l$
                Then the matrix $E_{kl}$ from Lemma~\ref{lem:Ekl} 
                satisfies that $\chi_{A+E_{kl}}(x) \neq \chi_{A}(x) \addconv
                \chi_{E_{kl}}(x)$.
                On the other hand all principal minors of $E_{kl}$ vanish and hence $E_{kl} \in \PB$.
		We showed that $A \notin \act{\mc{B}}_n$ and consequently $\act{\mc{B}}_n \subseteq \DD$.
		Together with the inclusion in the other direction, we have
		$$
		\act{\mc{B}}_n = \DD
		$$
		and the proof is complete.
	\end{proof}
	
	\begin{corollary}
		\label{cor:commuting}
		Let $\Sn{A} \subseteq \Cnn$ be a maximal family of commuting diagonalizable matrices and let $\PB \subseteq \Cnn$ be the family of principally balanced matrices.
 Then there exists a non-singular matrix $P$, such that $\Sn{A}$ and $P^{-1} {\PB} P$ form an additive complementary pair.
\end{corollary}
\begin{proof}
  Since the matrices in $\Sn{A}$ commute with each other, they are simultaneously diagonalizable (\cite{HJ13}, Theorem~1.3.21).
  That is, there exists a non-singular matrix $P$, such that for
  each $A \in \Sn{A}$ the conjugated matrix $D = P A P^{-1}$ is
  diagonal. 
  By the maximality of $\Sn{A}$ and the non-singularity of $P$, $\Sn{A} = P^{-1} \DD P$ is a subspace of dimension $n$.
  By Theorem~\ref{thm:addcomp_1}, for every $A \in \Sn{A}$ and $B \in \PB$, 
  $$
  \chi_{PAP^{-1}+B}(x) = \chi_{PAP^{-1}}(x) \addconv \chi_{B}(x)
  $$
  and the subsets $P \Sn{A} P^{-1}$ and $\PB$
  form an additive complementary pair.
  By similarity transformation, the same holds for $\Sn{A}$ and $P^{-1} {\PB} P$.
\end{proof}
	
	\subsection{Scalar matrices}
	Another additive complementary pair is the following.
	\begin{theorem}
		\label{thm:addcomp_2}
		Let $\Sn{S} \subseteq \Cnn$ be the family of $n \times n$ scalar matrices.
		Then $\Sn{S}$ and $\Cnn$ form an additive complementary pair.
              \end{theorem}

	\begin{proof}

        Let $A\in\Cnn$ be an arbitrary matrix.
        Then the zero matrix is in finite free position with $A$ because $\chi_0(x)=x^n$ is the neutral element for additive finite free convolution (Example~\ref{ex:unit}).
        It follows from Proposition~\ref{pr:closedness} that any scalar matrix is in finite free position with $A$ as well.
        Hence
		$$
		\Sn{S} \subseteq \act{\mc{M}}_n.
		$$
        For the converse, by Theorem~\ref{thm:addcomp_1}, since $\Cnn \supseteq \DD$ and $\Cnn \supseteq \PB$,
		$$
		\act{\mc{M}}_n \subseteq \act{\mc{D}}_n \cap \act{\mc{B}}_n = \PB \cap \DD = \Sn{S}.
		$$
		By both inclusions, we conclude that 
		$$
		\act{\mc{M}}_n = \Sn{S}.
		$$
		Clearly, since $\act{\mc{S}}_n = \biact{\mc{M}}_n \supseteq \Cnn$, we get that
		$$
		\act{\mc{S}}_n = \Cnn.
		$$
		It follows that $\Sn{S}$ and $\Cnn$
		form an additive complementary pair.
	\end{proof}

	\subsection{Triangular matrices}
	In this subsection we exhibit an additive complementary pair within the ring of upper triangular matrices.
	\begin{theorem}
          \label{thm:addcomptriangular}
          Let $\UT \subseteq \Cnn$ be the family of $n \times n$ upper
          triangular matrices and let $\PBUT \subseteq \Cnn$ be the family of
          $n \times n$ upper triangular matrices with constant diagonal, i.e.,
          matrices which can be written as the sum of a nilpotent and a scalar matrix.
          Then $\UT$ and $\PBUT$ form an additive complementary pair.

          Analogously, the lower triangular matrices $\LT$ and $\PBLT$ form an additive complementary pair.
    \end{theorem}
	\begin{proof}

		Let $T \in \UT$ and $C \in \PBUT$. Then $\chi_T (x), \chi_C (x)$ and $\chi_{T+C} (x)$ do not change if we replace them by their diagonals, i.e., if we set all off-diagonal entries to zero.
    	Then $T$ becomes diagonal and $C$  becomes a scalar matrix and thus principally balanced.
        By Theorem~\ref{thm:addcomp_1}, $\chi_{T+C}(x) = \chi_{T}(x) \addconv \chi_{C}(x)$.
        This shows that $\PBUT \subseteq \act{\mc{R}}_n$ and $\UT \subseteq \act{\PBUT}$.
		
		Next, we show that $\act{\PBUT} \subseteq \UT$.
        Suppose that $A=(a_{ij}) \notin \UT$, i.e., $a_{lk} \neq 0$ for some $(l,k)$ with $l > k$.
		Then the matrix $E_{kl}$ from Lemma~\ref{lem:Ekl} is in $\PBUT$ but $\chi_{A+E_{kl}}(x) \neq \chi_{A}(x) \addconv \chi_{E_{kl}}(x)$.
		It follows that $A \notin \act{\PBUT}$ and consequently that $\act{\PBUT} \subseteq \UT$.
		
		It remains to show that $\act{\mc{R}}_n \subseteq \PBUT$.
		Let $C \in \act{\mc{R}}_n$, then the argument of the previous paragraph shows that $C$ is upper triangular.
        Since $\DD\subseteq \UT$, it follows from Theorem~\ref{thm:addcomp_1} that $C$ is principally balanced 
        and thus $C\in\UT\cap\PB=\PBUT$.
	\end{proof}
	
	\begin{corollary}
		\label{cor:single_ev}
		If $A \in \Cnn$ is the sum of a nilpotent and a scalar matrix
		and $B \in \Cnn$ commutes with $A$ then, for any polynomials $p(x)$ and $q(x)$, $p(A)$ and $q(B)$ are in additive FFP.
	\end{corollary}
	\begin{proof}
		Since $A$ and $B$ commute, they are simultaneously triangularizable \cite[Theorem~2.3.3]{HJ13}, that is, for some unitary matrix $U$, $U^{*}AU$ and $U^{*}BU$ are upper triangular, with $U^{*}AU$ having a constant diagonal. The same holds for $p(A)$ and $q(B)$, so, by Theorem~\ref{thm:addcomptriangular}, $p(A)$ and $q(B)$ are in additive FFP.
	\end{proof}
	\begin{remark}
		Unlike diagonalization, commuting matrices are simultaneously triangularizable, but the converse is not true in general.
	\end{remark}
	
\subsection{Conclusion}
We conclude this section with the following observations.  Denote
$\mcb{V}=\mc{V}\cap\PB$ for any subset $\mc{V}\subseteq\Cnn$.
In particular, $\PB=\Snb{M}$.
\begin{enumerate}
 \item
  This notation is compatible with Theorem~\ref{thm:addcomptriangular}
  by the observation that a triangular matrix is
  principally balanced if and only if it has constant diagonal.
            
 \item
  In particular, scalar matrices are exactly the principally balanced
  diagonal matrices, $\Snb{D}=\PBDD$.
 \item
  The examples found so far can be summarized in the following ``recipe'':
  \begin{enumerate}[1.]
   \item
    Pick a pair $(\mc{V}_1,\mc{V}_2)$ among $(\DD,\Cnn)$, $(\UT,\UT)$ or
    $(\LT,\LT)$.
   \item
    Restrict one of the components to its principally balanced subset
    $\mcb{V}_i$.
\end{enumerate}
           \end{enumerate}

	\section{Moments and cumulants of sums of matrices in additive FFP}
	\label{sec:moments}
	In \cite{AP18} Arizmendi and Perales introduced \textbf{cumulants} for the  additive finite free convolution as the coefficients of a truncated $R$-transform, by showing that they satisfy the axiomatization of cumulants as defined in \cite{Lehner02}.
	Asymptotically, the finite free cumulants converge to free cumulants.
	Finite free cumulants are additive with respect to finite free convolution \cite{AP18},
	which implies that for matrices $A$ and $B$  in additive FFP we have
	\begin{equation}
		\label{eq:cumulants_additive}
		\kappa_i(A+B) = \kappa_i(A) + \kappa_i(B),
	\end{equation}
	where $\kappa_i(A)$ is the $i$-th order additive finite free cumulant of (the characteristic polynomial of) $A$. 
	A similar equality holds for independent as well as for free independent random variables. 
	
	Cumulants are tightly related to moments, with concrete moment--cumulant formulas, which allows passing from one representation to the other. Such formulas were introduced in \cite{AP18} in the finite free setting.
	This means that one can obtain formulas for the moments of $A+B$ in terms of the moments of $A$ and those of $B$ when $A$ and $B$ are in finite free position, as we show below.
	Conversely, when $A$ and $B$ satisfy these moment formulas, for $k=2, \ldots,n$, then they are in additive free position.
	
	The first \textbf{moment} (or normalized trace) of $A \in \Cnn$ is defined to be
	$$
	m_1(A) := \tr(A) = \frac{1}{n}\trace(A) = \frac{1}{n}\sum_{i=1}^{n} \lambda_i,
	$$
	the arithmetic mean of the eigenvalues $\lambda_i$ of $A$.
	In general, the $k$-th moment of $A$, $k \geq 1$, is
	$$
	m_k(A) := \tr(A^k) = \frac{1}{n} \sum_{i=1}^{n} \lambda_i^k.
	$$
	
	We make use of the following coefficient--moment formula of Lewin \cite{Lewin94} that is based on Newton's identities.
	Let $\chi_A (x) = \sum_{k=0}^{n} a_k x^{n-k}$. Then $a_0=1$ and, for $k=1,\ldots,n$,
	\begin{equation}
		\label{eq:Lewin}
		a_k = \sum_{\pi}\prod_{i=1}^{t}\frac{(-n \cdot m_{r_i}(A))^{s_i}}{r_i^{s_i} s_i !},
	\end{equation}
	where the summation is over all partitions $\pi$ of $k$ of the form
	$$
	k = \underbrace{r_1 + \cdots + r_1}_{s_1 \, \mathrm{ summands}} + \underbrace{r_2 + \cdots + r_2}_{s_2 \, \mathrm{ summands}} + \cdots + \underbrace{r_t + \cdots + r_t}_{s_t \, \mathrm{ summands}}
	=\sum_{i=1}^{t}s_i r_i,
	$$
	with $0 < r_1 < r_2 < \cdots < r_t$, and $r_i = r_i(\pi), s_i = s_i(\pi), t = t(\pi)$.
	
	\begin{theorem}
		\label{thm:momentsofsum}
		Let $A,B \in \Cnn$ be in additive FFP.
		Then the moments of $A+B$ can be expressed in terms of the moments of $A$ and the moments of $B$.	
	\end{theorem}
	\begin{proof}
		The proof is by induction. For $r=1$, we have
		$$m_1(A+B)=m_1(A)+m_1(B)$$ by linearity of the trace operator.
		Suppose that the statement holds for $1 \leq i < r$.
		Let $\chi_{A+B}(x) = \chi_{A}(x) \addconv \chi_{B}(x) = \sum_{i=0}^{n} c_i x^{n-i}$.
		By Newton's identity,
		$$
		-r c_r = \sum_{i=0}^{r-1} c_i \trace (A+B)^{r-i},
		$$
		thus,
		\begin{equation}
			\label{eq:Newton}
			m_r(A+B) = -\frac{r}{n} c_r - \sum_{i=1}^{r-1} c_i m_{r-i}(A+B).
		\end{equation}	 
		Since $A$ and $B$ are in additive FFP, we have, for $1 \leq i \leq r$,
		$$
		c_i = \sum_{l+j=i} \frac{{n-l \choose j}}{{n \choose j}} a_l b_j,
		$$
		and by \eqref{eq:Lewin}, the coefficients $a_l$ and $b_j$ can be written in terms of the moments of $A$, respectively $B$.
		It remains to express the moments $m_{r-i}(A+B)$ on the right hand side of \eqref{eq:Newton} through the moments of $A$ and $B$, which is assumed by the induction hypothesis.
		
		Equivalently, we can obtain a complex formula for $m_r(A+B)$ through the formulas in \cite{AP18}.
		By Equation~(4.5) in \cite{AP18}, $m_r(A+B)$ can be written in terms of the 
		cumulants $\kappa_1(A+B), \ldots, \kappa_r(A+B)$.
		Then, as shown in \cite{AP18}, Proposition~3.6, finite free cumulants are additive with respect to polynomial convolution, so that for each $i$, $\kappa_i(A+B) = \kappa_i(A) + \kappa_i(B)$.
		Finally, by Equation ~(4.4) in \cite{AP18}, each $\kappa_i(A)$ can be written as an expression in the moments $m_1(A), \ldots, m_i(A)$, and similarly for $\kappa_i(B)$.
	\end{proof}
	\begin{examples}
		We obtain simple formulas for the first, second and third moments of $A+B$ when $A,B \in \Cnn$ are in additive FFP. For the higher moments the formulas become more complex and, unlike the case $n \le 3$, they depend on the dimension of the matrices.
		Here are the first four moments.
		\begin{align*}
			m_1(A+B) &= m_1(A)+m_1(B);\\
			m_2(A+B) &= m_2(A)+2m_1(A)m_1(B)+m_2(B);\\
			m_3(A+B) &= m_3(A)+3m_2(A)m_1(B)+3m_1(A)m_2(B)+m_3(B);\\
			m_4(A+B) &= m_4(A) + 4m_3(A)m_1(B) + \ltfrac{2n}{n-1}m_2(A)m_1^2(B) \\ 
			&\quad + \ltfrac{4n-6}{n-1}m_2(A)m_2(B) - \ltfrac{2n}{n-1}m_1^2(A)m_1^2(B) \\
			&\quad + \ltfrac{2n}{n-1}m_1^2(A)m_2(B) + 4m_1(A)m_3(B) + m_4(B).
		\end{align*}
		
		We demonstrate the above formula for the second moment.
		\begin{align*}
			&m_2(A+B) = -\ltfrac{2}{n}c_2 - c_1m_1(A+B)\\
			&= -\ltfrac{2}{n}\Big(a_0b_2 +\ltfrac{n-1}{n}a_1b_1+ a_2b_0\Big) - \Big((a_0b_1 + a_1b_0)m_1(A+B)\Big) \\
			&= -\ltfrac{2}{n}\Big((\ltfrac{n^2}{2}m_1^2(B)-\ltfrac{n}{2}m_2(B)) + \ltfrac{n-1}{n}n^2m_1(A)m_1(B)\Big. \\
			&\quad + \Big.(\ltfrac{n^2}{2}m_1^2(A)-\ltfrac{n}{2}m_2(A))\Big) \\
			&\quad + \Big(n(m_1(A)+m_1(B))(m_1(A)+m_1(B))\Big) \\	
			&= -n m_1^2(B) + m_2(B) - 2(n-1)m_1(A)m_1(B) -n m_1^2(A) + m_2(A) \\
			&\quad + nm_1^2(B) + 2nm_1(A)m_1(B) + nm_1^2(A) \\
			&= m_2(A) + 2m_1(A)m_1(B) + m_2(B).
		\end{align*}
		
		By the linearity of the moment and the fact that $\trace(AB) = \trace(BA)$, it follows from the formulas for the second and third moments that when $A$ and $B$ are in FFP then
		\begin{align*}
			&m_1(AB) = m_1(A)m_1(B); \\
			&m_1(A B^2) + m_1(A^2 B) = m_1(A)m_2(B) + m_2(A)m_1(B).
		\end{align*}
	\end{examples}
	
	Let us now examine the case when $A$ is in additive FFP with itself.
	We denote by $\mc{P}(j)$ the set of partitions of the set $[j] = \{1,\ldots,j\}$.
	$\mc{P}(j)$ forms a lattice with a least element $0_j = \{\{1\},\{2\},\ldots,\{j\}\}$ and an upper element $1_j = \{\{1,2,\ldots,j\}\}$.
	When $\pi = \{V_1,\ldots,V_r\} \in \mc{P}(j)$ is a partition of $[j]$ with $r$ blocks then we set $|\pi|=r$ and denote by $\kappa_{\pi}(A)$ the product of the cumulants
	$$\kappa_{\pi}(A) := \kappa_{|V_1|}(A)\kappa_{|V_2|}(A)\cdots \kappa_{|V_r|}(A).$$
	The M\"{o}bius function $\mu(0_j,\pi)$ is
	$$\mu(0_j,\pi) = (-1)^{j-|\pi|} (2!)^{r_3} (3!)^{r_4} \cdots ((j-1)!)^{r_j},$$ where $r_i$ is the number of blocks of $\pi$ of size $i$. 
	\begin{theorem}
		\label{thm:selfaddfree}
		Let $A \in \Cnn$. Then $A$ is in additive FFP with itself if and only if 
		it is the sum of a nilpotent and a scalar matrix,
                i.e., the multiple of a unipotent matrix.
	\end{theorem}
	\begin{proof}
		When $A$ is the sum of a nilpotent and a scalar matrix
		then it is in additive FFP with itself by Corollary~\ref{cor:single_ev}.
		
		Suppose now that $A$ is in additive FFP with itself.
		By \eqref{eq:cumulants_additive}, $\kappa_j(2A) = 2\kappa_j(A)$, for $j \geq 1$.
		On the other hand, by the definition of a cumulant \cite{AP18}, $\kappa_j(2A) = 2^j \kappa_j(A)$. It follows that
		\begin{equation*}
			\kappa_j(A) = 0, \quad \textrm{ for } j \geq 2.
		\end{equation*}
		It is shown in \cite{AP18} that for $j \geq 1$, 
		\begin{equation}
			\label{eq:coeff_cumutant}
			m_j(A) =  \frac{(-1)^{j-1}}{n^{j+1}(j-1)!}\sum_{\pi \in \mc{P}(j)}
			n^{|\pi|}\mu(0_j,\pi)\kappa_{\pi}(A) \sum_{\rho : \rho \vee \pi = 1_j} n^{|\rho|}\mu(0_j,\rho).
		\end{equation}
		Since $\kappa_j(A) = 0$, for $j \geq 2$, then in the first summation in \eqref{eq:coeff_cumutant}, all summands vanish, except for the bottom partition $\pi = 0_j$, and then the second summation is over the set containing just the top partition $\rho=1_j$, resulting in
		\begin{align}
			\label{eq:coeff_cumulant_free}
			m_j(A) &= \frac{(-1)^{j-1}}{n^{j+1}(j-1)!}n^{|0_j|}\mu(0_j,0_j)\kappa_{0_j}(A) 
			n^{|1_j|}\mu(0_j,1_j) \nonumber \\
			&=\frac{(-1)^{j-1}}{n^{j+1}(j-1)!}n^j\kappa_1^j(A) n^1(-1)^{j-1}(j-1)! \\
			&=m_1^j(A). \nonumber
		\end{align}
		The last equation follows from $\kappa_1(A) = m_1(A) = \tr(A)$.
		
		The equations \eqref{eq:coeff_cumulant_free}, for $j=2,\ldots,n$, imply that the eigenvalues of $A$ equal each other, that is, $A$ is the sum of a nilpotent and a scalar matrix.
	\end{proof}
	
	\section{Matrices in multiplicative finite free position}
	\label{sec:multconv}
	In addition to additive convolution of polynomials, the notion of multiplicative convolution was introduced in \cite{MSS22}). We examine here matrices that are in multiplicative free position and obtain results that are similar to those in the additive case.
	\begin{definition}
		\label{def:multconv}
		Let $p(x) = \sum_{i=0}^{n} a_i x^{n-i}$, $q(x) = \sum_{i=0}^{n} b_i x^{n-i}$ be two polynomials of degree $n$ over $\CC$.
		The \textbf{multiplicative convolution} of $p(x)$ and $q(x)$, denoted $p(x) \multconv q(x)$, is
		\begin{equation}
			\label{eq:multconv}
			p(x) \multconv q(x)
			:= \sum_{k=0}^{n} \frac{(-1)^k }{{n \choose k}} a_k b_k x^{n-k}.
		\end{equation}
	\end{definition}
	
	
	The following theorem is the analogue of Theorem~\ref{thm:addorthogonal} in the multiplicative case.
	\begin{theorem}\cite{MSS22}
		\label{thm:multorthogonal}
		Let $A,B \in \Cnn$ be normal matrices.
		Then 
		\begin{equation}
			\label{eq:multorthogonal}
				\chi_{A}(x) \multconv \chi_{B}(x)
				= \int_{\mc{U}(n)} \chi_{AU^* BU}(x) \,dU
				= \frac{1}{2^nn!}\sum_{P\in \mc{P}^{\pm}(n)} \chi_{AP^{T} BP}(x),
			\end{equation}
			where the expectation is taken over the set of unitary matrices $\mc{U}(n)$ or the signed permutation matrices $\mc{P}^\pm(n)$.
	\end{theorem}
	
	\begin{definition}
		The matrices $A,B \in \Cnn$ are in \textbf{multiplicative finite free position} (or in multiplicative FFP) if
		$$
		\chi_{AB}(x) = \chi_{A}(x) \multconv \chi_{B}(x). 
		$$
		The families $\Sn{E},\Sn{F} \subseteq \Cnn$ are in multiplicative finite free position (in multiplicative FFP) if $\chi_{AB}(x) = \chi_{A}(x) \multconv \chi_{B}(x)$ for every $A \in \Sn{E}$ and $B \in \Sn{F}$. 
	\end{definition}
	
	\begin{remarks}
		\begin{enumerate}
			\item []
			\item As shown in \cite{Marcus21}, the property of being in multiplicative FFP can be expressed in terms of mixed discriminants.
			\item As in the additive case, $\chi_{A}(x) \multconv \chi_{B}(x)$ depends on the characteristic polynomials of the matrices and not on the matrices themselves, which is not the case for $\chi_{AB}(x)$. However, if $A$ and $B$ are in multiplicative FFP then so are $PAP^{-1}$ and $PBP^{-1}$, for any regular matrix $P$.
			\item For any two matrices $A,B \in \Cnn$, the leading monomial and the free coefficient of $\chi_{AB}(x)$ and $\chi_{A}(x) \multconv \chi_{B}(x)$ are the same: $x^n$ and $(-1)^n \det(AB)=(-1)^n \det(A)\det(B)$.
		\end{enumerate}
	\end{remarks}
	
	\subsection{Matrices of dimension $2 \times 2$}
	For matrices $A=(a_{ij}), B=(b_{ij}) \in \mathcal{M}_2$,
	we have
	\begin{align*}
		\chi_{AB}(x) &= x^2 - (a_{11} b_{11} + a_{12} b_{21} + a_{21} b_{12} + a_{22} b_{22})x \\ 
		&+ a_{11} a_{22} b_{11} b_{22} + a_{12} a_{21} b_{12} b_{21} - a_{11} a_{22} b_{12} b_{21} - a_{12} a_{21} b_{11} b_{22}
	\end{align*}
	and
	\begin{align*}
		\chi_{A}(x) \multconv \chi_{B}(x) &= x^2 - \frac{1}{2}(a_{11}+a_{22})(b_{11}+b_{22})x \\
		&+ (a_{11} a_{22} - a_{12} a_{21})(b_{11} b_{22} - b_{12} b_{21}).
	\end{align*}
	By comparing the coefficients of the powers of $x$, we get the following result.
	\begin{proposition}
		The matrices $A=(a_{ij}), B=(b_{ij}) \in \mathcal{M}_2$ are in multiplicative FFP if and only if
		$$
		(a_{11} - a_{22}) (b_{22} - b_{11}) = 2(a_{12} b_{21} + a_{21} b_{12}).
		$$
	\end{proposition}
	This is the same condition as in \eqref{eq:addFFP2x2}, although here the identical coefficients are of powers 2 and 0 of $x$, whereas in the additive case the identical  coefficients are of powers 2 an 1 of $x$.
	
	It follows, that also in the multiplicative case, a $2 \times 2$ matrix $A$ is in FFP with itself if and only if its two eigenvalues equal each other.
	The analogue of Proposition~\ref{pr:poly2x2} is the following proposition.
	\begin{proposition}
		Let  $A=(a_{ij}), B=(b_{ij}) \in \mathcal{M}_2$ be in multiplicative FFP. Then $p(A)$ and $q(B)$ are in multiplicative FFP, for any polynomials $p(x)$ and $q(x)$. Moreover, if $A$ is regular then $A^{-1}$ and $q(B)$ are in multiplicative FFP.
	\end{proposition}
	
	\subsection{Matrices of dimension $3 \times 3$}
	As we did in the additive case, we look now at diagonal matrices of dimension $3 \times 3$ that are in additive FFP.

	Let
	$$
	A = \begin{bmatrix*}[r]
		\lambda_1 & 0 & 0 \\
		0 &\lambda_2 & 0 \\
		0 & 0 & \lambda_3,
	\end{bmatrix*}, \qquad \qquad
	B = \begin{bmatrix*}[r]
		\mu_1 & 0 & 0 \\
		0 &\mu_2 & 0 \\
		0 & 0 & \mu_3
	\end{bmatrix*}.
	$$
	Then
	\begin{align*}
	\chi_{AB}(x) &= x^3 - (\lambda_1\mu_1+\lambda_2\mu_2+\lambda_3\mu_3) x^2 \\ 
	&+ (\lambda_1\lambda_2\mu_1\mu_2+\lambda_1\lambda_3\mu_1\mu_3+
	\lambda_2\lambda_3\mu_2\mu_3) x \\
	&- \lambda_1\lambda_2\lambda_3\mu_1\mu_2\mu_3
	\end{align*}
	and
	\begin{align*}
	\chi_{A}(x) &\multconv \chi_{B}(x) = x^3 - \frac{1}{3}(\lambda_1+\lambda_2+\lambda_3)(\mu_1+\mu_2+\mu_3) x^2 \\
	&+ \frac{1}{3}(\lambda_1 \lambda_2+\lambda_1 \lambda_3+\lambda_2 \lambda_3) 
	(\mu_1 \mu_2+\mu_1 \mu_3+\mu_2 \mu_3) x \\
	&- \lambda_1\lambda_2\lambda_3\mu_1\mu_2\mu_3.
	\end{align*}
	By comparing the coefficients of $x^2$ and $x$ in both expressions we get that the diagonal matrices $A$ and $B$ are in multiplicative FFP if and only if the following two equations are satisfied:
	\begin{equation}
		\label{eq:multFFP3x3a}
		\lambda_1\mu_2+\lambda_1\mu_3+\lambda_2\mu_1+\lambda_2\mu_3+
		\lambda_3\mu_1+\lambda_3\mu_2 = 2(\lambda_1\mu_1+\lambda_2\mu_2+\lambda_3\mu_3)
	\end{equation}
	and
	\begin{align}
		\label{eq:multFFP3x3b}
		&\lambda_1\lambda_2\mu_1\mu_3+\lambda_1\lambda_2\mu_2\mu_3+
		\lambda_1\lambda_3\mu_1\mu_2 +\lambda_1\lambda_3\mu_2\mu_3
		+\lambda_2\lambda_3\mu_1\mu_2+\lambda_2\lambda_3\mu_1\mu_3 \\
		&= 	2(\lambda_1\lambda_2\mu_1\mu_2+\lambda_1\lambda_3\mu_1\mu_3+\lambda_2\lambda_3\mu_2\mu_3). 
		\nonumber
	\end{align}
	We see that Equations \eqref{eq:multFFP3x3a} and \eqref{eq:addFFP3x3a} are the same but Equations \eqref{eq:multFFP3x3b} and \eqref{eq:addFFP3x3b} differ from each other.
	The following example demonstrates this difference between the additive and the multiplicative cases.	
	\begin{example}
		\label{ex:mFFPnotaFFP}
		Let
		$$
		A =	\begin{bmatrix*}[r]
			1 & 0 & 0 \\
			0 & 0 & 0 \\
			0 & 0 & 0
		\end{bmatrix*}, \qquad
		B = \begin{bmatrix*}[r]
			1 & 1 & 0 \\
			1 & 1 & 0 \\
			0 & 0 & 1
		\end{bmatrix*}.
		$$
		Then $\chi_{A}(x)=x^3-x^2$, $\chi_{B}(x)=x^3-3x^2+2x$, $\chi_{A+B}(x)=x^3-4x^2+4x+1$ and $\chi_{AB}(x)=x^3-x^2$.
		The additive and multiplicative convolutions are
		$\chi_{A}(x) \addconv \chi_{B}(x) = x^3-4x^2+4x-\frac{2}{3}$
		and $\chi_{A}(x) \multconv \chi_{B}(x) = x^3-x^2$. We see that unlike the situation in dimension $2 \times 2$, here the matrices $A$ and $B$ are in multiplicative FFP but not in additive FFP.
		
	\end{example}
	Let us now consider algebraic operations.
	\begin{example}
		\begin{enumerate}[(i)]
			\item []
			\item \emph{Adding a scalar does not preserve multiplicative FFP:}
			The matrices $A$ and $B$ from example~\ref{ex:mFFPnotaFFP} are in multiplicative FFP,
			but the pair $(I+A,B)$ is not.
			\item \emph{Squaring and inverting do not preserve multiplicative FFP:}
			The matrices
			$$
			A=\begin{bmatrix}
				1 & 0 & 0 \\
				0 & 2 & 0 \\
				0 & 0 & 3 
			\end{bmatrix}
			\qquad
			B=
			\begin{bmatrix}
				1 & -1 & 2 \\ 
				-1 & -2 & 1 \\ 
				2 & 1 & 1 
			\end{bmatrix}
			$$
			are in multiplicative FFP, but the pairs $(A^2,B)$ and $(A^{-1},B)$ are not.
		\end{enumerate}
	\end{example}
		
	\subsection{The lattice of multiplicative finite free affine algebraic sets of matrices}
	The lattice of multiplicative finite free (non-irreducible) affine algebraic sets is constructed in a way similar to the additive one.

	\begin{definition}
		Given a family $\mc{S} \subseteq \Cnn$, its \textbf{multiplicative finite free complement}, denoted $\mct{\mc{S}}$, is
		$$
		\mct{\mc{S}} = \{ B \in \Cnn \, : \, \chi_{AB}(x) = \chi_{A}(x) \multconv \chi_{B}(x), \mbox{ for all } A \in \mc{S} \}. 
		$$
		When $\mct{\mc{E}}_n = \Sn{F}$ and $\mct{\mc{F}}_n = \Sn{E}$, we say that $\Sn{E},\Sn{F}$ form a \textbf{multiplicative finite free complementary pair} (or a multiplicative complementary pair). 
	\end{definition}
	
	\begin{proposition}
		\label{pr:multclosedness}
		If $A$ and $B$ are in multiplicative FFP then so are $\lambda A$ and $B$, for every $\lambda \in \CC$.
	\end{proposition}
	\begin{proof}
		By definition, in both $\chi_{\lambda A}(x) \multconv \chi_{B}(x)$ and $\chi_{\lambda AB}(x)$, the coefficients of $x^{n-k}$ are multiplied by $\lambda ^k$ compared to  the coefficients of $\chi_{A}(x) \multconv \chi_{B}(x)$ and $\chi_{AB}(x)$.
	\end{proof}
	
	\subsection{Diagonal and principally balanced matrices}
	We show here that the pairs of complementary (non-irreducible) affine algebraic sets of matrices with respect to the additive convolution, discussed in Section~\ref{sec:addcomp}, are also complementary pairs in the multiplicative setting.
	\begin{theorem}
		\label{thm:multcomp_1}
		Let $\DD \subseteq \Cnn$ be the family of $n \times n$ diagonal matrices and let $\PB \subseteq \Cnn$ be the family of $n \times n$ principally balanced matrices.
		Then $\DD$ and $\PB$
		form a multiplicative complementary pair.	
	\end{theorem}
	\begin{proof}
		The theorem clearly holds for $n=1$, so let us assume that $n \geq 2$.
		First, we show that for arbitrary $D \in \DD$ and $B \in \PB$, the equality $\chi_{DB}(x) = \chi_{D}(x) \multconv \chi_{B}(x)$ holds.
		Let $D = \mathrm{diag} (d_1,\ldots,d_n)$,
		$$
		\chi_D (x) = \sum_{i=0}^{n} a_i x^{n-i}
		= \sum_{i=0}^{n} \bigg( (-1)^i \sum_{J \subseteq [n],\, |J|=i} d_J \bigg) x^{n-i},
		$$
		where $d_J = d_{j_1} \cdots d_{j_i}$ for $J=\{j_1, \ldots, j_i\}$.
		Let also
		$$\chi_B (x) = \sum_{i=0}^{n} b_i x^{n-i} = \sum_{i=0}^{n} (-1)^i {n \choose i} m_i x^{n-i},
		$$
		where $m_i$, $i=0,\ldots,n$, is the value of a principal minor of $B$ of order $i$.
		Since
		$$
		DB = 	\begin{bmatrix*}
			d_1 B_{1 \, \cdot} \\
			d_2 B_{2 \, \cdot}  \\
			\vdots \\
			d_n B_{n \, \cdot}
		\end{bmatrix*},
		$$
		where $B_{i \, \cdot}$ denotes the $i$-th row of $B$, we have
		\begin{align}
			\label{eq:chi_DB_1}
			\chi_{DB}(x) &=	\sum_{i=0}^{n} \bigg( (-1)^i \sum_{J \subseteq [n],\, |J|=i} \det((DB)_J) \bigg) x^{n-i}\nonumber \\
			&= \sum_{i=0}^{n} \bigg( (-1)^i \sum_{J \subseteq [n],\, |J|=i} d_J\det(B_J) \bigg) x^{n-i} \nonumber \\
			&= \sum_{i=0}^{n} \bigg( (-1)^i \sum_{J \subseteq [n],\, |J|=i} d_J m_i  \bigg) x^{n-i} \nonumber \\
			&= \sum_{i=0}^{n} a_i m_i x^{n-i} = \sum_{i=0}^{n} a_i \frac{(-1)^i b_i}{{n \choose i}} x^{n-i}.
		\end{align}
		By Definition~\ref{def:multconv}, this is exactly $\chi_{D}(x) \multconv \chi_{B}(x)$ and
		it follows that $\PB \subseteq \mct{\mc{D}}_n$ and $\DD \subseteq \mct{\mc{B}}_n$
		
		Next, we show that $\mct{\mc{D}}_n \subseteq \PB$.
		Let $B=(b_{ij}) \in \mct{\mc{D}}_n$.
		Let $K \subseteq [n]$, $|K|=k$, $0 < k < n$, and let $D(\lambda, K)=(d_{ij})$ be the parametric diagonal matrix with entries
		$$
		d_{ij} = \left\{ 
		\begin{array}{rcl}
			0 \quad &\mbox{for}& i \neq j, \\
			\lambda \quad &\mbox{for}& i=j \in K, \\
			1 \quad &\mbox{for}& i=j \notin K.
		\end{array}
		\right.
		$$
		\begin{align*}
			\chi_{D(\lambda,K)B}(x) &= \det (xI - (D(\lambda,K)B)) \\
			&= x^n + p_1(\lambda) x^{n-1} +\cdots+ p_{n-1}(\lambda) x + (-1)^n \det (D(\lambda,K)B),
		\end{align*}
		where $p_j(\lambda)$ is a polynomial in $\lambda$ of degree at most $\max (j,k)$.
		Let $p^{(i)}(x)$ be the $i$-th derivative of $p(x)$ with respect to $x$.
		Then 
		\begin{align*}
			(\chi_{D(\lambda,K)B})^{(n-k)}(0) &= (n-k)! p_k(\lambda) \\
			&= (-1)^k (n-k)! \lambda^k \det(B_K) + q(\lambda),
		\end{align*}
		where $q(\lambda)$ is a polynomial in $\lambda$ of degree less than $k$.
		We have
		\begin{equation}
			\label{eq:limproduct}
			\lim_{\lambda \to \infty} 
			\frac{(\chi_{D(\lambda,K)B})^{(n-k)}(0)}{(-1^k)(n-k)!\lambda^{k}} = \det(B_K).
		\end{equation}
		
		We now compute the same limit for $\chi_{D(\lambda,K)}(x) \multconv \chi_{B}(x)$.
		The characteristic polynomial of $D(\lambda,K)(x)$ is
		$$
		\chi_{D(\lambda,K)}(x) = \sum_{i=0}^{n} a_i x^{n-i} = (x-\lambda)^k (x-1)^{n-k}. 
		$$ 
		Then
		\begin{equation}
			\label{eq:ak}
			a_k = (-1)^k \lambda^k + g(\lambda),
		\end{equation}
		where $g(\lambda)$ is a polynomial in $\lambda$ of degree less than $k$.
		The characteristic polynomial of $B$ is
		$\chi_B (x) = \sum_{i=0}^{n} b_i x^{n-i}$ with
		\begin{equation}
			\label{eq:bk}
			b_k = (-1)^k \sum_{I \subseteq [n],\, |I|=k} \det(B_I).
		\end{equation}
		
		By \eqref{eq:multconv},
		\begin{align*}
			(\chi_{D(\lambda,K)} \multconv \chi_{B})^{(n-k)}(0)
			&= \left( \sum_{k=0}^{n} \frac{(-1)^k }{{n \choose k}} a_k b_k x^{n-k} \right)^{(n-k)} (0) \\
			&= \frac{(-1)^k (n-k)!}{{n \choose k}} a_k b_k
		\end{align*}
		Substituting for $a_k$ and $b_k$ the expressions in \eqref{eq:ak} and \eqref{eq:bk}, we get
		$$
		= \frac{(-1)^k (n-k)!}{{n \choose k}} ((-1)^k \lambda^k + g(\lambda)) (-1)^k \sum_{I \subseteq [n],\, |I|=k} \det(B_I).
		$$
		Computing the limit as in \eqref{eq:limproduct} gives
		\begin{equation}
			\label{eq:limmconv}
			\lim_{\lambda \to \infty} 
			\frac{(\chi_{D(\lambda,K)} \multconv \chi_{B})^{(n-k)}(0)}{(-1^k)(n-k)!\lambda^{k}} = 
			\frac{1}{{n \choose k}} \sum_{I \subseteq [n],\, |I|=k} \det(B_I).
		\end{equation}
		Comparing \eqref{eq:limproduct} with \eqref{eq:limmconv}, we get that
		$$
		\det(B_K) = \frac{1}{{n \choose k}} \sum_{I \subseteq [n],\, |I|=k}\det(B_I),
		$$
		and as in Theorem~\ref{thm:addcomp_1}, it follows that $B$ is principally balanced and therefore $\mct{\mc{D}}_n \subseteq \PB$.
		Together with the inclusion in the opposite direction, we have $\mct{\mc{D}}_n = \PB$.
		
		Finally, we show that $\mct{\mc{B}}_n \subseteq \DD$.
		So, assume by contradiction that $A=(a_{ij}) \in \mct{\mc{B}}_n$ and $A \notin \DD$.
		It follows that for some $(j_0,i_0)$, $j_0 \neq i_0$, $a_{j_0i_0} = a \neq 0$.
		Let $B=(b_{ij)})$ be the matrix $B=I_n + E_n(i_0j_0)$:
		$$
		b_{ij} = \left\{ 
		\begin{array}{rl}
			1 \quad &\mbox{for } i=j, \\
			1 \quad &\mbox{for } (i,j)=(i_0,j_0), \\
			0 \quad &\mbox{otherwise.}
		\end{array}
		\right.
		$$
		Clearly, $B \in \PB$.
		Let the characteristic polynomial of $A$ be $\chi_{A}(x) = \sum_{i=0}^{n} a_i x^{n-i}$ and that of $B$:
		$$
		\chi_{B}(x)  =  \sum_{i=0}^{n} b_i x^{n-i} = (x-1)^{n} = \sum_{i=0}^{n} {n \choose i} (-i)^i x^{n-i}.
		$$
		Since $\trace(AB) = \trace(A) + a$, the characteristic polynomial of $AB$ is
		\begin{equation}
			\label{eq:chi_AB}
			\chi_{AB}(x) = x^{n} + (a_1 - a) x^{n-1} + q(x), 
		\end{equation}
		where $q(x)$ is a polynomial in $x$ of degree less than $n-1$.
		
		The multiplicative convolution of $\chi_{A}(x)$ and $\chi_{B}(x)$ is
		\begin{align}
			\label{eq:chi_A_mconv_chi_B}
			\chi_{A}(x) \multconv \chi_{B}(x)
			&= \sum_{k=0}^{n} \frac{(-1)^k }{{n \choose k}} a_k b_k x^{n-k} \\
			&= \sum_{k=0}^{n} \frac{(-1)^k }{{n \choose k}} a_k {n \choose k} (-1)^k x^{n-k} \\
			&= \sum_{k=0}^{n} a_k x^{n-k} = \chi_{A}(x). 
		\end{align}
		We see that the coefficients of $x^{n-1}$ in \eqref{eq:chi_AB} and in \eqref{eq:chi_A_mconv_chi_B} are different from each other, implying that 
		$\chi_{AB}(x) \neq \chi_{A}(x) \multconv \chi_{B}(x)$.
		It follows that $\mct{\mc{B}}_n \subseteq \DD$ and with the inclusion in the opposite direction, we have $\mct{\mc{B}}_n = \DD$.
	\end{proof}	
	
	The next corollary is analogues to Corollary~\ref{cor:commuting} and is proved in a similar way.
	\begin{corollary}
		Let $\Sn{A} \subseteq \Cnn$ be a maximal family of commuting diagonalizable matrices and let $\PB \subseteq \Cnn$ be the family of principally balanced matrices. Then there exists a non-singular matrix $P$, such that $\Sn{A}$ and $P^{-1} {\PB} P$
		form a multiplicative complementary pair. \\
	\end{corollary}
	
	\subsection{Triangular matrices}
	As with the additive case, the upper triangular matrices and the upper triangular matrices with constant diagonal
	form a multiplicative complementary pair,
	and similarly for lower triangular matrices.
	\begin{theorem}
		\label{thm:multcomptriangular}
		Let $\UT \subseteq \Cnn$ be the family of $n \times n$ upper triangular matrices and let $\PBUT \subseteq \Cnn$ be the family of $n \times n$ upper triangular matrices with constant diagonal.
		Then $\UT$ and $\PBUT$
		form a multiplicative complementary pair.
	\end{theorem}
	\begin{proof}
		The lines of proof are analogous to those of Theorem~\ref{thm:addcomptriangular}.
		We only mention that when $A=(a_{ij}) \notin \UT$ with $a_{j_0i_0} = a \neq 0$, for some $j_0 > i_0$, then we define $B=(b_{ij}) \in \PBUT$ to be the matrix $B=I_n + E_n(i_0j_0)$ as in Theorem~\ref{thm:multcomp_1}.
		Then $\chi_{A}(x) \multconv \chi_{B}(x) = \chi_{A}(x)$ whereas $\chi_{AB}(x) \neq \chi_{A}(x)$.
	\end{proof}	
	
	\subsection{Scalar matrices}
	Also for scalar matrices the multiplicative convolution behaves analogously to the additive convolution.
	\begin{theorem}
		\label{thm:multcomp_2}
		Let $\Sn{S} \subseteq \Cnn$ be the family of $n \times n$ scalar matrices.
		Then $\Sn{S}$ and $\Cnn$
		form a multiplicative complementary pair.
	\end{theorem}
	\begin{proof}
		First, we show that when $A=(a_{ij}) \in \Cnn$ and $C = \mathrm{diag} (c,\ldots,c)$ then $\chi_{AC}(x) = \chi_{A}(x) \multconv \chi_{C}(x)$.
		As before, let
		$$
		\chi_A (x) = \sum_{i=0}^{n} a_i x^{n-i},
		$$
		and
		$$
		\chi_C (x) = \sum_{i=0}^{n} c_i x^{n-i} = 
		\sum_{i=0}^{n} (-1)^i {n \choose i} c^i x^{n-i}.
		$$
		Then
		$$
		\chi_{AC}(x) = \sum_{k=0}^{n} a_k c^k x^{n-k}
		$$
		and
		\begin{align}
			\chi_{A}(x) \multconv \chi_{C}(x)
			&= \sum_{k=0}^{n} \frac{(-1)^k }{{n \choose k}} a_k c_k x^{n-k} \nonumber \\
			&=	\sum_{k=0}^{n} \frac{(-1)^k }{{n \choose k}} a_k (-1)^k {n \choose k} c^k x^{n-k} \nonumber \\
			&= \sum_{k=0}^{n} a_k c^k x^{n-k} = \chi_{AC}(x). \nonumber
		\end{align}
		The rest of the proof is similar to that of Theorem~\ref{thm:addcomp_2}.
	\end{proof}
	
	\subsection{Matrices in multiplicative FFP with themselves}
	The following result shows the equivalence of additive and
        multiplicative FFP, see Theorem~\ref{thm:selfaddfree}.
	\begin{proposition}
		\label{pr:selfmultfree}
		Let $A \in \Cnn$. Then $A$ is in multiplicative FFP with itself if and only if 
		it is the sum of a nilpotent and a scalar matrix,
                i.e., the multiple of a unipotent matrix.
	\end{proposition}
	\begin{proof}
		By similarity transformation, we can assume, without loss of generality, that $A$ is in upper triangular form with its eigenvalues $\lambda_1, \lambda_2,\ldots,\lambda_n$ on the diagonal. 
		Let $\chi_A (x) = \sum_{k=0}^{n} a_k x^{n-k}$, then
		$$
		a_k = (-1)^k \sum_{\substack{I = \{i_1,i_2,\ldots,i_k\} \\
				1 \leq i_1<i_2<\cdots<i_k\leq n}}
		\lambda_{i_1}\lambda_{i_2}\cdots \lambda_{i_k}.
		$$
		Hence
		\begin{equation}
			\label{eq:charpolyAsquared}
			\chi_{A^2} (x) = \sum_{k=0}^{n} 
			(-1)^k \sum_{\substack{I = \{i_1,i_2,\ldots,i_k\} \\
					1 \leq i_1<i_2<\cdots<i_k\leq n}}
			\lambda^2_{i_1}\lambda^2_{i_2}\cdots \lambda^2_{i_k} \,
			x^{n-k}.
		\end{equation}
		On the other hand,
		\begin{align}
			\label{eq:selfmultconv}
			\chi_A(x) \multconv \chi_A(x)
			&= \sum_{k=0}^{n} \frac{(-1)^k }{{n \choose k}} a^2_k x^{n-k} \nonumber \\
			&= \sum_{k=0}^{n} 
			\frac{(-1)^k }{{n \choose k}} \Big(\sum_{\substack{I = \{i_1,i_2,\ldots,i_k\} \\
					1 \leq i_1<i_2<\cdots<i_k\leq n}}
			\lambda_{i_1}\lambda_{i_2}\cdots \lambda_{i_k}\Big)^2
			x^{n-k}.
		\end{align}
		By comparing the corresponding coefficients in \eqref{eq:charpolyAsquared}  and \eqref{eq:selfmultconv}, for $k=1,\ldots,n-1$, we get that  $\lambda_1 = \lambda_2 = \cdots = \lambda_n$, that is, $A$ is the sum of a nilpotent and a scalar matrix.
		
		In the other direction, when 
		$A$ is the sum of a nilpotent and a scalar matrix,
		then $A$ is similar to an upper triangular matrix with constant diagonal and, by Theorem~\ref{thm:multcomptriangular}, $A$ is in multiplicative FFP with itself.
	\end{proof}
	
	\subsection{Moments and cumulants of products of matrices  in multiplicative FFP}
	Moments and cumulants for the multiplicative convolution of polynomials
        were handled by Arizmendi et al.\ \cite{AGP21}. Among others, they derived a formula for the finite free cumulants of the multiplicative convolution of polynomials $p$ and $q$ in terms of the finite free cumulants of $p$ and of $q$. They also were able to express the moments of the empirical root distribution of $p \multconv q$ in terms of the finite free cumulants of $p$ and the moments of $q$. 
	As shown in \cite{AGP21}, finite free multiplicative convolution converges to free multiplicative convolution.
	
	As in the additive case, when $A$ and $B$ are in multiplicative FFP, we can compare the coefficients of $\chi_{A}(x) \multconv \chi_{B}(x)$ with the corresponding coefficients of $\chi_{AB}(x)$ and derive formulas for the moments of $AB$ in terms of those of $A$ and $B$.
	Next, we derive these formulas for the first and second moments.
	\begin{proposition}
		Let $A,B \in \Cnn$ be in multiplicative FFP. Then
		\begin{align}
			m_1(AB) &= m_1(A) m_1(B), \\
			m_2(AB) &= \frac{n}{n-1} [m_2(A)m_1^2(B) + m_1^2(A)m_2(B) - m_1^2(A)m_1^2(B)] \\
			&-\frac{1}{n-1}m_2(A)m_2(B).  \nonumber
		\end{align}
		\label{pr:moments_product}
	\end{proposition}
	\begin{proof}
		Let $\chi_A (x) = \sum_{k=0}^{n} a_k x^{n-k}$ and  $\chi_B (x) = \sum_{k=0}^{n} b_k x^{n-k}$.
		The coefficient of $x^{n-1}$ in $\chi_{A}(x) \multconv \chi_{B}(x)$ is
		\begin{align}
			\label{eq:coeff1multconv}
			-\frac{1}{n} a_1 b_1 = -\frac{1}{n} \trace(A)\trace(B)
			= -n \cdot m_1(A)m_1(B),
		\end{align}
		whereas the coefficient of $x^{n-1}$ in $\chi_{AB}(x)$ is
		\begin{equation}
			\label{eq:coeff1mult}
			-\trace(AB) = -n \cdot m_1(AB).
		\end{equation}
		The formula for the first moment then follows. \\ \\
		
		For the formula of the second moment, we compare the coefficients of $x^{n-2}$ in $\chi_{A}(x) \multconv \chi_{B}(x)$ and in $\chi_{AB}(x)$. 
		In $\chi_{A}(x) \multconv \chi_{B}(x)$ it is
		$\frac{1}{{n \choose 2}} a_2 b_2$, which equals by \eqref{eq:Lewin}
		\begin{align}
			\label{eq:coeff2multconv}
			&\frac{n}{2(n-1)} [n^2 \cdot m_1^2(A)m_1^2(B) + m_2(A)m_2(B) 
			- n \cdot m_1^2(A)m_2(B) \nonumber \\ &- n \cdot m_2(A)m_1^2(B)].
		\end{align}
		The coefficient of $x^{n-2}$ in $\chi_{AB}(x)$ is
		\begin{equation}
			\label{eq:coeff2mult}
			\frac{\trace^2(AB)}{2} - \frac{\trace(AB)^2}{2} = \frac{n^2}{2} \cdot m_1^2(A)m_1^2(B) - \frac{n}{2} \cdot m_2(AB).
		\end{equation}
		The formula for the second moment then follows.
	\end{proof}
	
	Comparing other coefficients of powers of $x$ do not result in simple and nice formulas for higher moments of $AB$.
	
	As for cumulants, we can apply the formula of Arizmendi, Garza-Vargas and Perales \cite{AGP21} for finite free cumulants of finite free multiplicative convolutions, which implies that for matrices $A,B \in \Cnn$ that are in multiplicative FFP we have
	\begin{equation}
		\kappa_j(AB) =  \frac{(-1)^{j-1}}{n^{j+1}(j-1)!}\sum_{\substack{\sigma,\tau \in \mc{P}(j) \\ \sigma \vee \tau = 1_j}}
		n^{|\sigma|+|\tau|}\mu(0_j,\sigma)\mu(0_j,\tau)\kappa_{\sigma}(A)\kappa_{\tau}(B).
	\end{equation}
	The first three finite free cumulants are the following.
	\begin{align}
		\kappa_1(AB) &= \kappa_1(A) \kappa_1(B). \\
		\kappa_2(AB) &= -\frac{1}{n} \kappa_2(A) \kappa_2(B) + \kappa_2(A) \kappa_1^2(B) + \kappa_1^2(A) \kappa_2(B). \\
		\kappa_3(AB) &= \frac{2}{n^2}\kappa_3(A)\kappa_3(B) -\frac{3}{n}\kappa_3(A)\kappa_2(B)\kappa_1(B) + \kappa_3(A)\kappa_1^3(B) \\
		&\; \; -\frac{3}{n}\kappa_2(A)\kappa_1(A)\kappa_3(B) + 3\kappa_2(A)\kappa_1(A)\kappa_2(B)\kappa_1(B) + \kappa_1^3(A)\kappa_3(B). \nonumber
	\end{align}
		
	\bibliographystyle{plain}
	\bibliography{Convolvent}
\end{document}